\documentclass[10pt]{article}
\usepackage{amsmath, amssymb, amsthm, amsfonts, cases}
\usepackage{mathrsfs}
\usepackage{url}
\usepackage{authblk}
\usepackage[usenames]{color}
\usepackage{geometry}
\usepackage{indentfirst}
\usepackage{bbm}
\allowdisplaybreaks

\theoremstyle{plain}
\newtheorem{thm}{Theorem}[section]

\newtheorem{lem}{Lemma}[section]

\theoremstyle{definition}
\newtheorem{defn}{Definition}[section]
\newtheorem{ass}{Assumption}[section]
\newtheorem{rmk}{Remark}[section]

\makeatletter\@addtoreset{equation}{section} \makeatother

\begin{document}		
	\title{Infinite Horizon Fully Coupled Nonlinear Forward-Backward Stochastic Difference Equations and Their Application to LQ Optimal Control Problems
		\thanks{Q. Meng was supported by the Key Projects of Natural Science Foundation of Zhejiang Province (No.
			Z22A013952) and the
			National Natural Science Foundation of China ( No.12271158 and No. 11871121). X. Ma is supported by the Postgraduate Research
			and Innovation Project of Huzhou University (No.2024KYCX66). Xun Li is supported by RGC of Hong Kong grants 15221621, 15226922 and 15225124, Research Centre for Quantitative Finance grant P0042708, and partially from PolyU 4-ZZP4.}}
	
	\date{}
	
	\author[a]{Xinyu Ma}
	\author[b]{Xun Li}
	\author[a]{Qingxin Meng\footnote{Corresponding author.
			\authorcr
			\indent E-mail address: mxyzfd@163.com (X. Ma), li.xun@polyu.edu.hk (X. Li), mqx@zjhu.edu.cn (Q.Meng)}}

	\affil[a]{\small{Department of Mathematical Sciences, Huzhou University, Zhejiang 313000,PR  China}}
	\affil[b]{\small{Department of Applied Mathematics, The Hong Kong Polytechnic University, Hong Kong 999077, PR China}}
	
	\maketitle

	\begin{abstract}
		This paper focuses on the study of infinite horizon fully coupled nonlinear forward-backward stochastic difference equations (FBS$\bigtriangleup$Es). Firstly, we establish a pair of priori estimates for the solutions to forward stochastic difference equations (S$\bigtriangleup$Es) and backward stochastic difference equations (BS$\bigtriangleup$Es), respectively. Then, to achieve broader applicability, we utilize a set of domination-monotonicity conditions that are more lenient than standard assumptions. Using these conditions, we apply continuation methods to prove the unique solvability of infinite horizon fully coupled FBS$\bigtriangleup$Es and derive a set of solution estimates.
		Furthermore, our results have considerable implications for a variety of related linear quadratic (LQ) problems, especially when the stochastic Hamiltonian system is consistent with FBS$\bigtriangleup$Es satisfying the introduced domination-monotonicity conditions. Thus, by solving the associated stochastic Hamiltonian system, we explicitly characterize the unique optimal control. This is the first work establishing solvability of fully coupled nonlinear FBS$\bigtriangleup$Es
		under domination-monotonicity conditions in infinite horizon discrete-time setting.
	\end{abstract}

	
	\section{ Introduction}
	
	Backward stochastic differential equations (BSDEs) have drawn attention due to their standard structure and extensive applications. It is well known that the BSDEs trace their origin back to the stochastic control problem in Bismut \cite{2} in 1973. Since the work of Pardoux-Peng \cite{11} in 1990, problems related to BSDEs have witnessed significant development.
	By integrating forward stochastic differential equations (SDEs) and BSDEs into a single system, a class of coupled forward-backward stochastic differential equations (FBSDEs) is obtained. The system is called fully coupled FBSDEs if the initial and terminal values of the state are coupled.
	The coupled  FBSDEs were first introduced in the seminal work of Antonelli \cite{1}, where the existence of solutions was established over sufficiently small time horizons via a contraction mapping argument. Nevertheless, a counterexample was furnished as well to demonstrate that the same inference might not stand for a long-duration time interval despite the presence of the Lipschitz condition. In recent years, a large number of studies on solvability have been carried out, and currently, there exist two main approaches. The first approach is known as the four-step scheme approach. It integrates partial differential equations and probabilistic methods, with the stipulation that all coefficients in the equations must be non-stochastic and the coefficients of the diffusion terms should be non-degenerate. The method was first proposed by Ma-Protter-Yong \cite{9} in 1994 and was subsequently further developed
	by Delarue \cite{3} in 2002 and Zhang \cite{26} in 2006, among others. The second method, namely the method of continuation, originates from the stochastic Hamiltonian system that emerges from the process of solving control problems. It attains the solvability of the random coefficients of the FBSDE over an arbitrary time period by imposing specific monotonicity conditions on the coefficients. The method was first put forward by Hu-Peng \cite{5} in 1995, and afterwards, it was further advanced by Peng-Wu \cite{12} in 1999 and Yong \cite{24} in 1997.

	In this paper, we consider the following fully coupled infinite horizon nonlinear  FBS$\bigtriangleup $E :
	\begin{equation} \label{eq:1.1}
		\left\{\begin{aligned}
			& x_{k+1}  =b\left(k, \theta_k\right)+\sigma(k, \theta_k) \omega_k, \\
			& y_k = -f\left(k+1, \theta_k\right),\quad k \in \{0,1,2,\cdots\}, \\
			& x_0  = \Lambda (y_0),
		\end{aligned}\right.
	\end{equation}
	where $\theta_k:=(x_k,y'_{k+1},z'_{k+1}) $, $z_{k+1}:=y_{k+1}w_k,\ y'_{k+1}:=\mathbb E[y_{k+1}|\mathcal{F}_{k-1}],\ z'_{k+1}:=\mathbb E[z_{k+1}|\mathcal{F}_{k-1}] $. Additionally, $x$ and $y$ take values in some given Euclidean space $\mathbb{H}$.
	Moreover, $b$, $\sigma$ and $f$ are three given mappings which value in $\mathbb{H}$ and $\left\{w_k\right\}_{k\in \mathbb{T}_{\infty}}$ is a $\mathbb{R}$-valued martingale difference sequence. In order to simplify the notation,  for any $ \theta=(x,y',z'),$
	we define
	
	\begin{equation}
		\Gamma (k,\theta )=(f(k+1,\theta) ,b(k,\theta ),\sigma (k,\theta))
	\end{equation}
	and
	\begin{equation}
		\langle\Gamma(k, \theta), \theta\rangle=\langle f(k+1, \theta), x\rangle+\langle b(k, \theta), y'\rangle+\langle\sigma(k, \theta), z'\rangle .
	\end{equation}
	By this means, all coefficients of FBS$\bigtriangleup$E (\ref{eq:1.1}) are gathered together in $(\Lambda,\Gamma)$. For convience, we define FBS$\triangle$E (\ref{eq:1.1}) with coefficients $(\Lambda,\Gamma)$ as FBS$\triangle$E $(\Lambda,\Gamma)$.
	
	It is well known that the solvability of FBS$\bigtriangleup$Es  forms the basis of the stochastic optimal control theory as well as the stochastic game problem \cite{19}, \cite{14}. For instance, in \cite{14}, the time-inconsistent stochastic linear-quadratic control problem was investigated by making use of FBS$\bigtriangleup$Es. The study in \cite{17} on the $H^\infty$ control problem with preview was carried out by working out a Hamilton-Jacob system, which is, in essence, made up of FBS$\bigtriangleup$Es. \cite{20} conducted a study on the open-loop Stackelberg game problem, making use of the solution of FBS$\bigtriangleup$Es. The relevant application of continuous-time FBSDEs in the stochastic control problem can be located in \cite{12} as well as its references. In this paper, we study the solvability problem of the fully coupled infinite horizon nonlinear  FBS$\bigtriangleup $E\eqref{eq:1.1}, which is a very meaningful problem.
	
	There have been certain achievements in the research on linearly fully coupled FBS$\bigtriangleup$Es. Xu-Zhang-Xie \cite{21}\cite{22} investigated a class of fully coupled FBS$\bigtriangleup$Es from the perspective of linear quadratic (LQ) optimal control. In the equations they studied, the backward variables took the form of conditional expectations and fulfilled the condition that the coefficient matrices of the backward variables were degenerate. Regarding the fully coupled infinite horizon FBS$\bigtriangleup$Es, Xu-Xie-Zhang \cite{23} in 2017 and Song-Liu \cite{15} in 2020 utilized a generalized Riccati approach to systematically depict and explicitly formulate the adaptive solution. It should be noted that in 2016, Zhang-Qi \cite{27} promoted the progress of discrete-time FBS$\bigtriangleup$Es as they derived the maximum principle for discrete-time mean-field LQ control problems for the first occasion.
	
	Nevertheless, there are relatively few existing theoretical results for solvability problems related to nonlinearities, especially those in the infinite horizon case.
	In 2024, Ji-Liu \cite{6} carried out a study on the solvability problem of a class of coupled FBS$\bigtriangleup$Es and managed to acquire an existence and uniqueness theorem for nonlinear FBS$\bigtriangleup$Es within the framework of the traditional monotonicity condition.
	With respect to FBSDEs, motivated by different stochastic linear LQ problems, Yu \cite{25} put forward a more comprehensive domination-monotonicity framework. By employing the continuation method, a unique solvability outcome and a couple of estimates for general coupled FBSDEs were achieved. It should be noted that the domination-monotonicity framework constructed by Yu \cite{25} in 2022, through the introduction of diverse matrices, matrix-valued random variables and matrix-valued stochastic processes, was applicable to a wide range of stochastic LQ problems. This novel framework encompasses the prior traditional monotonicity condition and extends the traditional monotonicity condition. Furthermore, Ji-Liu \cite{6} stated that the generator $f$ in BS$\bigtriangleup$E  is only reliant on the solution at the time $k+1$ of (\ref{eq:1.1}). Consequently, the continuation method can be applied to the discrete-time case.
	
	In this paper, a class of discrete-time fully coupled nonlinear FBS$\bigtriangleup$E (\ref{eq:1.1}) is established by means of a discrete method based on the LQ problem. The domination-monotonicity framework proposed by Yu \cite{25} is discretized. The method of continuation is then applied to acquire a pair of estimates for FBS$\bigtriangleup$E (\ref{eq:1.1}) and to prove their unique solvability. We also consider two LQ problems in the application of the results. It is worth noting that there exist several distinctions between prior studies and our work. First of all, the monotonicity conditions we utilize are not the same as those in previous research. The condition presented in this paper is more comprehensive and offers a wider scope in application compared to the traditional one. By enhancing the Lipschitz condition and relaxing the monotonicity condition, it becomes more conducive to the investigation of LQ problems. Secondly, the systems under consideration are distinct. Based on the LQ optimal control problem, this paper puts forward a nonlinear system. In this system, the equations possess inverted variables in the form of conditional expectations. Moreover, these equations ensure that the coefficient matrices of the inverted variables are degenerate, thereby strictly ensuring the measurability of the state. It ought to be emphasized that, in the context of the infinite horizon, a concept of stabilizability was frequently investigated to guarantee the well-posedness of the associated cost functional (\cite{4}\cite{8}\cite{12}\cite{15}). In infinite horizon continuous FBSDEs, as discussed in   \cite{99} and \cite{18}, introducing a weight function, such as \( e^{2Kt} \), is a common method to ensure the existence and uniqueness of the solution. The weight function helps control the growth of the solution, preventing divergence over an infinite time horizon, and ensures that the solution remains bounded. It also facilitates energy estimates and the derivation of inequalities, typically using an exponential decay function like \( e^{2K t} \), which ensures the solution decays over time and avoids explosion. Additionally, the weight function allows the influence of different parts of the equation to vary over time, reducing the impact of the solution in the distant future. By transforming the problem into a finite weighted space, this approach makes it possible to handle the challenges of infinite time horizons effectively, ensuring that the solution behaves in a well-defined and reasonable manner. In summary, the introduction of a weight function, as noted in \cite{99} and \cite{18}, is crucial for guaranteeing the existence and uniqueness of solutions in infinite horizon problems.
	
	In this paper, we develop the weight function approach for discrete-time systems, extending the method introduced in   \cite{99} and  \cite{18} for continuous systems and investigates the existence and uniqueness of solutions for discrete-time FBS$\bigtriangleup $E \eqref{eq:1.1} over an infinite-time horizon.  This method proves to be instrumental in ensuring the stability and tractability of solutions, especially in the context of infinite time horizons. Furthermore, the paper applies this approach to the infinite horizon discrete-time stochastic linear quadratic (LQ) optimal control problem, where the goal is to determine the optimal control policy that minimizes a given cost functional. {\color{blue} To derive the optimal control policy explicitly, we employ the stochastic maximum principle (SMP)(\cite{36}\cite{62}\cite{60}\cite{35}\cite{37}\cite{61}), which extends the continuous-time framework of Wu \cite{50} and Shi-Wu \cite{51} to discrete-time systems. It provides necessary and sufficient conditions for optimality via the Hamiltonian system. Specifically, the optimal control is characterized by minimizing the Hamiltonian function associated with the cost functional, while the adjoint equations govern the evolution of the costate variables. This framework aligns with our domination-monotonicity conditions, ensuring the measurability and stability of the optimal control.} The introduction of a weight function allows us to adaptively manage the influence of various components of the system, ensuring that the cost functional remains well-defined and that the solution to the control problem exists and is unique.
	By leveraging the weight function to handle the infinite time horizon, we establish corresponding theoretical results for the existence and uniqueness of solutions in both the FBSDEs and the stochastic LQ control problem. {\color{blue}Notably, the weight function technique also facilitates the application of the SMP in infinite horizon settings. By ensuring the decay of solutions (via terms like $e^{2Kt}$), the weighted space framework guarantees that the Hamiltonian function and adjoint equations remain well-defined over infinite time, thereby validating the SMP-based optimality conditions.} The results demonstrate that the use of weight functions not only ensures the well-posedness of these problems but also offers a systematic way to address the difficulties introduced by the infinite horizon. This approach lays the foundation for further research into stochastic control problems in infinite time domains, providing insights into the behavior of such systems under long-term dynamic optimization. In recent years, theoretical research on stochastic control problems over infinite time horizons has achieved significant progress. Notably, the seminal work of Ji-Zhang \cite{7} laid a crucial foundation for  recursive control problems within discrete-time frameworks.
	
	The rest of this paper is organized as follows. In Section \ref{sec:2}, we begin by defining a number of notations and providing two lemmas related to infinite horizon S$\bigtriangleup $Es and BS$\bigtriangleup $Es. In Section \ref{sec:3}, we utilize the continuation method in the research of FBS$\bigtriangleup $E \eqref{eq:1.1} within the domination-monotonicity conditions. In Theorem \ref{thm:4.1}, we display the unique solvability result and a pair of estimates. In Section 4, applying the obtained results, we carry out an investigation into two LQ optimal control problems and verify the existence and uniqueness of the optimal control.

	\section{Notations and Preliminaries }\label{sec:2}
	\subsection{Notations}	
	Let $\mathbb{T}_{\infty}$ and $\mathbb{T}_{N}$ denote the sets $\{0,1,2,\cdots\}$ and $\{0,1,2,\cdots,N\}$ respectively, where $N$ is a given positive integer representing the time horizon. Consider a complete filtered probability space $(\Omega, \mathcal{F}, \mathbb{F}, \mathbb{P})$, where the filtration $\mathbb{F}=\left\{\mathcal{F}_k : k=0,1,2, \cdots\right\}$. Let $\left\{w_k\right\}_{k\in \mathbb{T}_{\infty}}$ be a $\mathbb{R}$-valued martingale difference sequence defined on a probability space $(\Omega, \mathcal{F}, P)$, and
	$$
	\mathbb{E}\left[w_{k+1} \mid \mathcal{F}_k\right]=0, \quad \mathbb{E}\left[(w_{k+1})^{2} \mid \mathcal{F}_k\right]=1.
	$$
	In this paper, we suppose that $\mathcal{F}_k$ is the $\sigma$-algebra generated by $\left\{x_0, w_l, l=\right.$ $0,1, \ \cdots, k\}$. For convenience, $\mathcal{F}_{-1}$ denotes $\left \{ \emptyset,\Omega  \right \}$.

	Let $\mathbb{R}^n$ be the $n$-dimensional Euclidean space with the norm $|\cdot|$ and the inner product $\langle\cdot,\cdot\rangle$. Let $\mathbb{S}^n$ be the set of all symmetric matrices in $\mathbb{R}^{n\times n}$. Let $\mathbb{R}^{n\times m}$ be the collection of all $n\times m$ matrices with the norm $|A|=\sqrt{\textrm{tr}(AA^\top)}$ and the inner product $	\left\langle A,B\right\rangle = \textrm{tr}(AB^\top)$,  for any $ A,B\in \mathbb{R}^{n\times m}$.

	Let $\mathbb H$ be a Hilbert space with norm $\|\cdot\|_\mathbb H$, $\rho\in\mathbb{R}$, then we present some notations as follows:

	$\bullet$ $L^{2} _{\mathcal{F}_{k} }(\Omega ; \mathbb H)$: the set of all $\mathbb H$-valued  $\mathcal{F}_{k}$-measurable random variables $\eta$ satisfying
	
	\begin{equation}
		\|\eta\|_{L^{2} _{\mathcal{F}_{k} }(\Omega ;\mathbb{H})} :=\Big[\mathbb{E}\|\eta \|_\mathbb H^{2} \Big]^{\frac{1}{2}}<\infty .\nonumber
	\end{equation}
	
	$\bullet$ $L^{\infty} _{\mathcal{F}_{k} }(\Omega ;\mathbb H)$: the set of all $\mathbb H$-valued
	$\mathcal{F}_{k}$-measurable essentially bounded variables.
	
	$\bullet$ $L_{\mathbb{F}}^{2,\rho}\left(\mathbb{T}_{\infty}; \mathbb{H}\right)$: the set of all
	$\mathbb{H}$-valued  stochastic process $g=\{(g_0,g_1,g_2,\cdots)|g_k$ is $\mathcal{F}_{k-1}$-measurable, $k=0,1,2, \cdots$\} satisfying
	
	\begin{equation}
		\|g(\cdot)\|_{L_{\mathbb{F}}^{2,\rho}\left(\mathbb{T}_{\infty}; \mathbb{H}\right)} :=\bigg [\mathbb{E}\bigg (\sum_{k=0}^{\infty} e^{-\rho k}\|g_k\|_{\mathbb{H}}^{2}\bigg ) \bigg]^{\frac{1}{2}}<\infty.\nonumber
	\end{equation}
	
	$\bullet$ $L_{\mathbb{F}}^{\infty}(\mathbb{T}_{\infty}; \mathbb{H})$: the set of all
	$\mathbb{H}$-valued  essentially bounded stochastic processes $h(\cdot)=\{(h_0,h_1,h_2,\cdots)|h_k$ is $\mathcal{F}_{k-1}$-measurable, $ k=0,1,2, \cdots$\}.
	
	$\bullet$ $\mathbb{U}$ : the set of all admissible controls
	$$
	\mathbb{U}:=L_{\mathbb{F}}^{2,\rho}\left(\mathbb{T}_{\infty}; \mathbb{R}^{m}\right).
	$$

	$\bullet$ $N^{2,\rho}(\mathbb{T}_{\infty};\mathbb{R}^{2n} ):= L_{\mathbb{F}}^{2,\rho}\left(\mathbb{T}_{\infty}; \mathbb{R}^n\right)\times L_{\mathbb{F}}^{2,\rho}\left(\mathbb{T}_{\infty}; \mathbb{R}^n\right)$. For any $(x(\cdot),y(\cdot)) \in N^{2,\rho}(\mathbb{T}_{\infty};\mathbb{R}^{2n} )$, its norm is given by
	\begin{equation}
		\begin{aligned}
			\|(x(\cdot),y(\cdot)) \|_{N^{2,\rho}(\mathbb{T}_{\infty};\mathbb{R}^{2n} )}:= \left \{ \mathbb{E}\bigg [\displaystyle\sum_{k=0}^{\infty}e^{-\rho k}\left( |x_k|^{2}+|y_k|^{2}\right)  \bigg ] \right \}^{\frac{1}{2}}.\nonumber
		\end{aligned}
	\end{equation}
	
	$\bullet$ $N^{2,\rho}(\mathbb{T}_{\infty};\mathbb{R}^{3n} ):= L_{\mathbb{F}}^{2,\rho}\left(\mathbb{T}_{\infty}; \mathbb{R}^n\right)\times L_{\mathbb{F}}^{2,\rho}\left(\mathbb{T}_{\infty}; \mathbb{R}^n\right)\times L_{\mathbb{F}}^{2,\rho}\left(\mathbb{T}_{\infty}; \mathbb{R}^n\right)$.
	For any $\phi(\cdot) =(\varphi(\cdot)  ,\psi(\cdot) ,\gamma(\cdot) ) \in N^{2,\rho}(\mathbb{T}_{\infty};\mathbb{R}^{3n} )$, its norm is given by
	\begin{equation}
		\begin{aligned}
			\|\phi(\cdot) \|_{{N}^{2,\rho}(\mathbb{T}_{\infty};\mathbb{R}^{3n} )}:=\left \{ \mathbb{E}\bigg [\sum_{k=0}^{\infty}e^{-\rho k}\left( |\varphi_k|^{2}+|\psi_k|^{2}+|\gamma _k|^{2}\right)  \bigg ]   \right \}^{\frac{1}{2}}.
		\end{aligned}\nonumber
	\end{equation}

	$\bullet$ $\mathcal{H}^{2,\rho}(\mathbb{T}_{\infty}):=\ L^{2} _{\mathcal{F}_{-1} }(\Omega ; \mathbb{R}^{n})\times N^{2,\rho}(\mathbb{T}_{\infty};\mathbb{R}^{3n} )$. For any $(\xi  ,\phi(\cdot))\in \mathcal{H}^{2,\rho}(\mathbb{T}_{\infty})$, its norm is given by
	\begin{equation}
		\begin{aligned}
			\|(\xi ,\phi(\cdot) )\|_{\mathcal{H}^{2,\rho}(\mathbb{T}_{\infty})}:=\left \{ \|\xi \|^{2}_{ L^{2} _{\mathcal{F}_{-1} }(\Omega ; \mathbb{R}^{n})}+\|\phi(\cdot) \|^{2}_{N^{2,\rho}(\mathbb{T}_{\infty};\mathbb{R}^{3n} )} \right \}^{\frac{1}{2}}.\nonumber
		\end{aligned}
	\end{equation}
	\subsection{Infinite horizon S$\bigtriangleup $E }	
	Some basic results on infinite horizon S$\bigtriangleup $Es are presented below. As a first step, we study the following infinite horizon S$\bigtriangleup $E:
	\begin{equation}\label{eq:2.1}        		
		\left\{\begin{aligned}               	 	
			x_{k+1}  =&b\left(k, x_{k}\right)+\sigma(k, x_k) \omega_k, \\
			x_0  =&\eta , \quad k\in \mathbb{T}_{\infty}.
		\end{aligned}\right.
	\end{equation}
	For convenience, we denote  S$\triangle$E \eqref{eq:2.1} with coefficients $(b,\sigma,\eta)$ as S$\triangle$E $(b,\sigma,\eta).$	Here we provide the following definition of the solution to infinite horizon S$\triangle$E $(b,\sigma,\eta)$.
	\begin{defn}  For some $ \rho\in \mathbb R,$  a $\mathbb{R}^n$-valued  stochastic process  $x(\cdot)\in  L^{2,\rho}_{\mathbb{F} }(\mathbb{T}_{\infty};\mathbb{R}^n )$ is called a  solution of  S$\triangle$E $(b,\sigma,\eta)$  if (i) $b(\cdot, x(\cdot)) \in  L^{2,\rho}_{\mathbb{F} }(\mathbb{T}_{\infty};\mathbb{R}^n )$, $\sigma(\cdot, x(\cdot)) \in  L^{2,\rho}_{\mathbb{F} }(\mathbb{T}_{\infty};\mathbb{R}^n )$; (ii)  for any $k\in \mathbb{T}_{\infty}$,  \eqref{eq:2.1} holds.
	\end{defn}
	The coefficients $(b,\sigma,\eta )$ are assumed to satisfy the following conditions:
	
	\begin{ass}\label{ass:2.1}
		$\eta \in L^{2} _{\mathcal{F}_{-1} }(\Omega ; \mathbb{R}^{n})$ and $(b,\sigma)$ are two  given random mappings
		\begin{equation}
			\begin{aligned}
				&b:\Omega\times\mathbb{T}_{\infty}\times  \mathbb{R}^n  \to \mathbb{R}^n ,\\
				&\sigma :\Omega\times\mathbb{T}_{\infty}\times  \mathbb{R}^n  \to \mathbb{R}^{n}
			\end{aligned}\nonumber
		\end{equation}
		satisfying:
		\\(i) For any $x \in \mathbb{R}^n$, $b(k,x) $ and $\sigma(k,x) $ are $\mathcal{F}_{k-1}$-measurable, $ k \in \mathbb{T}_{\infty}$;
		\\(ii)  $ b(\cdot, 0) \in  L^{2,\rho}_{\mathbb{F} }(\mathbb{T}_{\infty};\mathbb{R}^n )$  and $\sigma(\cdot, 0)\in  L^{2,\rho}_{\mathbb{F} }(\mathbb{T}_{\infty};\mathbb{R}^n );$
		\\(iii) The mappings $b$ and $\sigma$  are uniformly Lipschitz continuous with respect to $x$, i.e., there exists a constant $L_{1}>0$ such that the inequalities
		\begin{equation}
			|b(k,x)-b(k,\bar{x})|\le L_1 |x-\bar{x}|	,	\quad \quad\forall k \in \mathbb{T}_{\infty}\nonumber
		\end{equation}
		and
		\begin{equation}
			|\sigma(k,x)-\sigma (k,\bar{x})|\le L_1|x-\bar{x}|	, \quad \quad \forall k \in \mathbb{T}_{\infty}\nonumber
		\end{equation}
		hold for any $x,\bar{x}\in \mathbb{R}^n$.
		
	\end{ass}

	\begin{lem}\label{lem:2.1}
		Suppose that Assumption~\ref{ass:2.1} holds, and the constant \(\rho\) satisfies \(\rho > \ln(4L_1^2)\).
		{\color{blue}Assume further that the S$\triangle$E \((b, \sigma, \eta)\)  admits a solution
			\(x(\cdot)  \in L^{2,\rho}_{\mathbb{F}}(\mathbb{T}_{\infty}; \mathbb{R}^n)\).}
		Then there exists a constant \(\mathcal{C} > 0\), depending only on \(L_1\) and \(\rho\), such that
		\begin{equation}\label{eq:2.3}
			\sum_{k=0}^{\infty} e^{-\rho k}\,
			\mathbb{E}\bigl[\lvert x_{k}\rvert^{2}\bigr]
			\;\le\;
			\mathcal{C}\,
			\mathbb{E}\Biggl[
			\lvert \eta\rvert^{2}
			+ \sum_{k=0}^{\infty} e^{-\rho k}\,\bigl(\lvert b(k,0)\rvert^{2} + \lvert \sigma(k,0)\rvert^{2}\bigr)
			\Biggr].
		\end{equation}
		Moreover, if \(\bar x(\cdot)\) is the solution of  S$\triangle$E \((\bar b,\bar\sigma,\bar\eta)\) satisfying the same assumptions, then
		\begin{equation}\label{eq:2.4}
			\sum_{k=0}^{\infty} e^{-\rho k}\,
			\mathbb{E}\bigl[\lvert x_{k}-\bar x_{k}\rvert^{2}\bigr]
			\;\le\;
			\mathcal{C}\,
			\mathbb{E}\Biggl[
			\lvert \eta-\bar\eta\rvert^{2}
			+ \sum_{k=0}^{\infty} e^{-\rho k}\,
			\bigl(\lvert b(k,\bar x_{k})-\bar b(k,\bar x_{k})\rvert^{2}
			+\lvert \sigma(k,\bar x_{k})-\bar\sigma(k,\bar x_{k})\rvert^{2}\bigr)
			\Biggr].
		\end{equation}
	\end{lem}

	\begin{proof}
		we first prove the estimate \eqref{eq:2.3}. According to the S$\bigtriangleup$E equation, \(x_k\) satisfies
		\[
		x_{k+1} = b(k, x_k) + \sigma(k, x_k) \omega_k.
		\]
		Taking the square and taking expectation, we have
		\[
		\mathbb{E}[|x_{k+1}|^2] = \mathbb{E}[|b(k, x_k) + \sigma(k, x_k) \omega_k|^2].
		\]
		By expanding the square, we can get
		\[
		\mathbb{E}[|x_{k+1}|^2] = \mathbb{E}[|b(k, x_k)|^2] + \mathbb{E}[|\sigma(k, x_k) \omega_k|^2] + 2 \mathbb{E}[\langle b(k, x_k), \sigma(k, x_k) \omega_k \rangle].
		\]
		Since \(\mathbb{E}[\omega_k|\mathcal{F}_{k-1}] = 0\) and \(\mathbb{E}[|\omega_k|^2|\mathcal{F}_{k-1}] = 1\), we have
		\begin{equation*}
			\begin{split}
				\mathbb{E}\left[|\sigma(k, x_k) \omega_k|^2\right]=	\mathbb{E}\left[\mathbb{E}[|\sigma(k, x_k) \omega_k|^2|\mathcal{F}_{k-1}]\right]=\mathbb{E}\left[|\sigma(k, x_k)|^2\mathbb{E}[|\omega_k|^2|\mathcal{F}_{k-1}] \right]=\mathbb{E}[|\sigma(k, x_k)|^2],
			\end{split}
		\end{equation*}
		\begin{equation*}
			\mathbb{E}[\langle b(k, x_k), \sigma(k, x_k) \omega_k \rangle=\mathbb{E}\left[\mathbb{E}[b^{\top}(k, x_k)\sigma(k, x_k) \omega_k|\mathcal{F}_{k-1}]\right]=\mathbb{E}\left[b^{\top}(k, x_k)\sigma(k, x_k)\mathbb{E}[ \omega_k|\mathcal{F}_{k-1}]\right]=0.
		\end{equation*}
		Therefore\[
		\mathbb{E}[|x_{k+1}|^2] = \mathbb{E}[|b  (k, x_k)|^2] + \mathbb{E}[|\sigma(k, x_k)|^2].
		\]
		By using the Lipschitz condition, it is easy to see that
		
		\begin{equation}\label{eq:2.5}
			\begin{aligned}
				\mathbb{E}[|x_{k+1}|^2] = \mathbb{E}[|b(k, x_k)|^2]+ \mathbb{E}[|\sigma(k, x_k)|^2]
				\leq  2\mathbb{E}[|b(k,0)|^2] + 2\mathbb{E}[|\sigma(k,0)|^2] + 4L_{1}^2 \mathbb{E}[|x_k|^2].
			\end{aligned}
		\end{equation}
		Multiplying this inequality by \(e^{-\rho(k+1)}\) and summing from \(k=0\) to \(\infty\), we get
		\[
		\sum_{k=0}^{\infty} e^{-\rho(k+1)}\,
		\mathbb{E}[\lvert x_{k+1}\rvert^{2}]
		\;\le\;
		2\,e^{-\rho}
		\sum_{k=0}^{\infty} e^{-\rho k}\,
		\mathbb{E}\bigl[\lvert b(k,0)\rvert^{2} + \lvert\sigma(k,0)\rvert^{2}\bigr]
		\;+\;
		4\,L_{1}^{2}\,e^{-\rho}\,A,
		\]
		where
		\[
		A \;=\;
		\sum_{k=0}^{\infty} e^{-\rho k}\,
		\mathbb{E}\bigl[\lvert x_{k}\rvert^{2}\bigr].
		\]
		Notice that
		\[
		A
		\;=\;
		\mathbb{E}\bigl[\lvert x_{0}\rvert^{2}\bigr]
		\;+\;
		\sum_{k=0}^{\infty} e^{-\rho(k+1)}\,
		\mathbb{E}\bigl[\lvert x_{k+1}\rvert^{2}\bigr].
		\]
		Substituting the previous bound yields
		\[
		A \;\le\;
		\mathbb{E}\bigl[\lvert\eta\rvert^{2}\bigr]
		\;+\;
		2\,e^{-\rho}
		\sum_{k=0}^{\infty} e^{-\rho k}\,
		\mathbb{E}\bigl[\lvert b(k,0)\rvert^{2} + \lvert\sigma(k,0)\rvert^{2}\bigr]
		\;+\;
		4\,L_{1}^{2}\,e^{-\rho}\,A.
		\]
		Since the condition \(\rho>\ln(4L_{1}^{2})\) guarantees \(1 - 4L_{1}^{2}\,e^{-\rho}>0\), we can rearrange to obtain
		\[
		A
		\;\le\;
		\frac{1}{1 - 4L_{1}^{2}e^{-\rho}}
		\,\mathbb{E}\bigl[\lvert\eta\rvert^{2}\bigr]
		\;+\;
		\frac{2\,e^{-\rho}}{1 - 4L_{1}^{2}e^{-\rho}}
		\sum_{k=0}^{\infty} e^{-\rho k}\,
		\mathbb{E}\bigl[\lvert b(k,0)\rvert^{2} + \lvert\sigma(k,0)\rvert^{2}\bigr].
		\]
		Setting
		\[
		\mathcal{C}
		=
		\frac{1 + 2\,e^{-\rho}}{1 - 4\,L_{1}^{2}\,e^{-\rho}}
		\]
		gives exactly the desired estimate \eqref{eq:2.3}.
		
		For estimate \eqref{eq:2.4}, define  \(\widehat x_{k} = x_{k} - \bar x_{k}\).  Subtracting the two equations yields
		\[
		\widehat x_{k+1}
		= b\bigl(k,x_{k}\bigr) - \bar b\bigl(k,\bar x_{k}\bigr)
		+ \bigl[\sigma(k,x_{k}) - \bar\sigma(k,\bar x_{k})\bigr]\,\omega_{k}.
		\]
		By the same squaring, expectation, Lipschitz, weighting and summation argument as above, one derives \eqref{eq:2.4}.   The proof is complete.
	\end{proof}

	{\color{blue}
		\begin{thm}\label{thm:2.1}
			Under Assumption \ref{ass:2.1}, choose the constant $\rho$ so that $\rho > \ln(4L^2_1)$. Then the S$\bigtriangleup $E $(b,\sigma,\eta)$ has a unique solution $x(\cdot)\in L^{2,\rho}_{\mathbb{F} }(\mathbb{T}_{\infty};\mathbb{R}^n )$.
		\end{thm}
		
		\begin{proof}
			
			Define the space
			\begin{equation}\label{eq:Hspace}
				\mathcal{H} := L^{2,\rho}_{\mathbb{F}}(\mathbb{T}_{\infty};\mathbb{R}^n),
			\end{equation}
			equipped with the norm
			\begin{equation}
				\|x(\cdot)\|_{\rho}^2 := \sum_{k = 0}^{\infty} e^{-\rho k} \mathbb{E}\bigl[|x_{k}|^2\bigr].
			\end{equation}
			Here, $(\mathcal{H}, \|\cdot\|_{\rho})$ is a Hilbert space. Define the operator $T:\mathcal{H}\to\mathcal{H}$ as:
			\[
			(Tx)_{0} = \eta, \quad (Tx)_{k+1} = b(k,x_{k}) + \sigma(k,x_{k})w_{k}, \quad k\geq 0.
			\]
			A fixed point of $T$ corresponds to a solution of \eqref{eq:2.1}.
			
			\noindent\textbf{Step 1. $T$ maps $\mathcal{H}$ into itself.}
			First, consider the initial term $(Tx)_0 = \eta$. By assumption on the initial data, we have the finite expectation:
			\[
			\mathbb{E}[|(Tx)_0|^2] = \mathbb{E}[|\eta|^2] < \infty.
			\]
			Next, for the recursive terms with $k \geq 0$, we expand the operator action:
			\[
			(Tx)_{k+1} = b(k,x_k) + \sigma(k,x_k)w_k.
			\]
			Taking the square expectation yields:
			\begin{align*}
				\mathbb{E}[|(Tx)_{k+1}|^2]
				&= \mathbb{E}[|b(k,x_k)|^2] + \mathbb{E}[|\sigma(k,x_k)|^2].
			\end{align*}
			Applying the linear growth conditions $|b(k,x)| \leq |b(k,0)| + L_1|x|$ and $|\sigma(k,x)| \leq |\sigma(k,0)| + L_1|x|$, we bound each term:
			\begin{align*}
				\mathbb{E}[|b(k,x_k)|^2] &\leq 2\mathbb{E}[|b(k,0)|^2] + 2L_1^2\mathbb{E}[|x_k|^2], \\
				\mathbb{E}[|\sigma(k,x_k)|^2] &\leq 2\mathbb{E}[|\sigma(k,0)|^2] + 2L_1^2\mathbb{E}[|x_k|^2].
			\end{align*}
			Summing over $k \geq 0$ with weights $e^{-\rho(k+1)}$, we obtain:
			\begin{align*}
				\sum_{k=0}^\infty e^{-\rho(k+1)}\mathbb{E}[|(Tx)_{k+1}|^2]
				\leq e^{-\rho}\bigg\{2\sum_{k=0}^\infty e^{-\rho k}\mathbb{E}\left[|b(k,0)|^2 + |\sigma(k,0)|^2\right]+ 4L_1^2 \sum_{k=0}^\infty e^{-\rho k}\mathbb{E}[|x_k|^2]\bigg\} .
			\end{align*}
			The first summation on the right-hand side of the inequality converges because $b(\cdot,0), \sigma(\cdot,0) \in \mathcal{H}$, while the second term is finite due to $x(\cdot) \in \mathcal{H}$.
			Then we conclude that $Tx(\cdot) \in \mathcal{H}$.

			\medskip
			
			\noindent\textbf{Step 2. $T$ is a strict contraction.}
			
			Fix $x(\cdot),\bar x(\cdot)\in\mathcal{H}$ and set $g(\cdot) = Tx(\cdot)$, $\bar g(\cdot) = T\bar x(\cdot)$. For each $k\ge0$,
			\[
			g_{k+1}-\bar g_{k+1}
			=
			\bigl[b(k,x_{k})-b(k,\bar x_{k})\bigr]
			+
			\bigl[\sigma(k,x_{k})-\sigma(k,\bar x_{k})\bigr]\,w_{k}.
			\]
			Taking the conditional expectation given $\mathcal{F}_{k-1}$ and then the total expectation, we obtain
			\[
			\mathbb{E}\bigl[\,|g_{k+1}-\bar g_{k+1}|^{2}\bigr]
			=
			\mathbb{E}\Bigl[\,|b(k,x_{k})-b(k,\bar x_{k})|^{2}
			+|\sigma(k,x_{k})-\sigma(k,\bar x_{k})|^{2}\Bigr]
			\le 2L_{1}^{2}\,\mathbb{E}\bigl[|x_{k}-\bar x_{k}|^{2}\bigr].
			\]
			We now sum with the weights $e^{-\rho m}$. By definition,
			\[
			\|g(\cdot)-\bar g(\cdot)\|_{\rho}^{2}
			=
			\sum_{m=0}^{\infty} e^{-\rho m}
			\,\mathbb{E}\bigl[\,|g_{m}-\bar g_{m}|^{2}\bigr]
			=
			\sum_{k=0}^{\infty} e^{-\rho(k+1)}
			\,\mathbb{E}\bigl[\,|g_{k+1}-\bar g_{k+1}|^{2}\bigr].
			\]
			Substituting the one - step bound, we get
			\begin{align*}
				\|g(\cdot)-\bar g(\cdot)\|_{\rho}^{2}
				&\le
				\sum_{k=0}^{\infty} e^{-\rho(k+1)}
				\bigl[\,2L_{1}^{2}\,\mathbb{E}|x_{k}-\bar x_{k}|^{2}\bigr]\\
				&=
				2L_{1}^{2}\,e^{-\rho}
				\sum_{k=0}^{\infty} e^{-\rho k}
				\,\mathbb{E}\bigl[|x_{k}-\bar x_{k}|^{2}\bigr]\\
				&=
				2L_{1}^{2}\,e^{-\rho}\,\|x(\cdot)-\bar x(\cdot)\|_{\rho}^{2}.
			\end{align*}
			Then
			\[
			\|T x(\cdot) - T\bar x(\cdot)\|_{\rho}
			\le  \bigl(\sqrt{2}\,L_{1}\,e^{-\rho/2}\bigr)\,\|x(\cdot)-\bar x(\cdot)\|_{\rho}.
			\]
			Since $\rho>\ln(4L_{1}^{2})$, we have $\sqrt{2}\,L_{1}\,e^{-\rho/2}<1$. So, the operator $T$ is a strict contraction.
			
			\medskip
			
			By the Banach fixed-point theorem, $T$ has a unique fixed point $x^{*}(\cdot)\in\mathcal{H}$, which is the solution of S$\triangle$E $(b,\sigma,\eta)$. For uniqueness, suppose \(x(\cdot)\) and \(\bar{x}(\cdot)\) are two solutions. Applying \eqref{eq:2.4} with \((\bar{b}, \bar{\sigma}, \bar{\eta}) = (b, \sigma, \eta)\), the right-hand side vanishes, implying \(x(\cdot) = \bar{x}(\cdot)\).
	\end{proof}}

	\subsection{Infinite horizon  BS$\bigtriangleup $E}		
	We consider the infinite horizon BS$\bigtriangleup $E as follows:
	\begin{equation}\label{eq:3.1}
		\begin{aligned}
			&y_k =f\left(k+1,y'_{k+1},z'_{k+1}\right),
		\end{aligned}
	\end{equation}
	where  $z_{k+1}=y_{k+1}w_k,\ y'_{k+1}=\mathbb E[y_{k+1}|\mathcal{F}_{k-1}],\ z'_{k+1}=\mathbb E[z_{k+1}|\mathcal{F}_{k-1}] $. For convenience, we define BS$\triangle$E \eqref{eq:3.1} with coefficient $f$ as BS$\triangle$E $(f)$.
	The definition of the solution to infinite horizon  BS$\bigtriangleup $E $(f)$ is introduced as follows:
	\begin{defn}
		For some $ \rho\in \mathbb R,$ a $\mathbb{R}^n$-valued stochastic process $y(\cdot)\in L^{2,\rho}_{\mathbb{F} }(\mathbb{T}_{\infty};\mathbb{R}^n )$ is called a  solution of   BS$\bigtriangleup $E $(f)$ if
			(i) $f(\cdot,y'(\cdot),z'(\cdot))\in L_{\mathbb{F}}^{2,\rho}(\mathbb{T}_{\infty};\mathbb{R}^n)$;
			(ii) for any non-negative integer  $N$,
		\begin{equation}\label{eq:3.2}
			\left\{\begin{aligned}
				&y_k =f\left(k+1,y'_{k+1},z'_{k+1}\right), \\
				\\&y_N  =y_N,\quad k \in \mathbb{T}_{N}.
			\end{aligned}\right.
		\end{equation}
		
	\end{defn}
	The coefficient $f$ is assumed to satisfy the following assumptions:
	\begin{ass}\label{ass:3.1}
		$f$ is a given  mapping:	$ f: \Omega \times \mathbb{T}_{\infty}  \times \mathbb{R}^n \times \mathbb{R}^{n}\to \mathbb{R}^n$
		satisfying: \\
		(i) For any $y',z' \in \mathbb{R}^n$, $f(k+1,y',z') $ is $\mathcal{F}_{k-1}$-measurable, $k \in \mathbb{T}_{\infty}$;
		\\(ii)  $f(\cdot , 0, 0)\in  L^{2,\rho}_{\mathbb{F} }(\mathbb{T}_{\infty};\mathbb{R}^n )$;
		\\(iii) the mapping $f$ is uniformly Lipschitz continuous with respect to ($y',z'$), i.e.,  there exists a constant $L_{2}>0$ such that the inequality
		\begin{eqnarray}
			\begin{aligned}
				|f&(k+1 , y', z')-f(k+1 ,\bar{y}' , \bar{z}')|\le L_2\big(|y'-\bar{y}' |+|z'-\bar{z}' |\big), \quad k \in \mathbb{T}_{\infty}.\nonumber
			\end{aligned}
		\end{eqnarray}
		holds for any $y',\bar{y}',z',\bar{z}'\in \mathbb{R}^{n}$.

	\end{ass}

	\begin{lem}\label{lem:2.2}
		Let Assumption \ref{ass:3.1} hold, and we continue to assume that  $
		\rho < -\ln(6L^2_2)$. Suppose $y(\cdot)\in L^{2,\rho}_{\mathbb{F} }(\mathbb{T}_{\infty};\mathbb{R}^n )$ is a solution to BS$\bigtriangleup $E $(f)$. Then we have the following estimate:
		\begin{eqnarray}\label{eq:3.3}
			\sum_{k=0}^{\infty} e^{-\rho k} \mathbb{E}  \left[| y_k|^2 \right]
			\le \mathcal{C}\displaystyle\sum_{k=0}^{\infty}e^{-\rho (k+1)}\mathbb E\left[|f(k+1,0,0)|^2
			\right],
		\end{eqnarray}
		where $\mathcal{C}$ is a positive constant depending only on the constants $L_{2}$ and $\rho$. Furthermore, let $\bar{f}$ be another coefficient satisfying Assumption \ref{ass:3.1} and suppose that $\bar{y}(\cdot)\in L^{2,\rho}_{\mathbb{F} }(\mathbb{T}_{\infty};\mathbb{R}^n ) $ is a solution to BS$\bigtriangleup $E $(\bar{f})$. Then we have the following estimate:
		\begin{eqnarray}\label{eq:3.4}
			\begin{aligned}
				\sum_{k=0}^{\infty}e^{-\rho k}\mathbb{E}\left[|y_{k}-\bar{y}_{k}|^{2}\right]
				\le\  \mathcal{C}\displaystyle\sum_{k=0}^{\infty}e^{-\rho (k+1)}\mathbb E\left[
				|f(k+1,\bar{y}'_{k+1},\bar{z}'_{k+1})-\bar f(k+1,\bar{y}'_{k+1},\bar{z}'_{k+1})|^2\right],
			\end{aligned}
		\end{eqnarray}
		where  $\bar{z}_{k+1}=\bar{y}_{k+1}w_k,\ \bar{y}'_{k+1}=\mathbb E[\bar{y}_{k+1}|\mathcal{F}_{k-1}],\ \bar{z}'_{k+1}=\mathbb E[\bar{z}_{k+1}|\mathcal{F}_{k-1}] $ and $\mathcal{C}$ is a positive constant depending only on the constants $L_{2}$ and $\rho$.
	\end{lem}
	\begin{proof}
		
		The application of Jensen's inequality and the H$\ddot{o}$lder's inequality allows us to derive that
		
		\begin{equation}\label{eq:3.5}
			\begin{aligned}
				{\mathbb{E}}\left[\left|y_k^{\prime}\right|^2\right]  ={\mathbb{E}}\bigg[ |{\mathbb{E}}\left[y_k \mid\mathcal{F}_{k-2}\right]|^2 \bigg]
				\leq {\mathbb{E}}\left[{\mathbb{E}}\left[|y_k|^2\mid \mathcal{F}_{k-2}\right]\right]
				={\mathbb{E}}\left[|y_k|^2\right]
			\end{aligned}
		\end{equation}
		and
		\begin{equation}\label{eq:3.6}
			\begin{aligned}
				{\mathbb{E}}\left[\left|z_k^{\prime}\right|^2\right] & ={\mathbb{E}}\bigg[|{\mathbb{E}}\left[y_k w_{k-1} \mid \mathcal{F}_{k-2}\right]|^2 \bigg]\\
				& \left.\leq {\mathbb{E}}\bigg[{\mathbb{E}}\left[|y_k|^2\mid \mathcal{F}_{k-2}\right] \cdot {\mathbb{E}}\left[w_{k-1}^2 \mid \mathcal{F}_{k-2}\right]\right] \\
				& ={\mathbb{E}}\left[{\mathbb{E}}\left[|y_k|^2 \mid \mathcal{F}_{k-2}\right]\right] \\
				& ={\mathbb{E}}\left[|y_k|^2\right].
			\end{aligned}
		\end{equation}
		In a similar manner, estimates for for $y_k'-\bar{y}'_k $ and $z_k'-\bar{z}'_k$ can be obtained as follows:
		\begin{equation}\label{eq:3.7}
			\begin{aligned}
				{\mathbb{E}}\left[\left|y_k^{\prime}-\bar{y}'_{k}\right|^2\right]  ={\mathbb{E}}\bigg[ |{\mathbb{E}}\left[y_k-\bar{y}_{k} \mid\mathcal{F}_{k-2}\right]|^2 \bigg]
				\leq {\mathbb{E}}\left[{\mathbb{E}}\left[|y_k-\bar{y}_{k}|^2\mid \mathcal{F}_{k-2}\right]\right]
				={\mathbb{E}}\left[|y_k-\bar{y}_{k}|^2\right],
			\end{aligned}
		\end{equation}
		\begin{equation}  \label{eq:3.8}
			\begin{aligned}
				\mathbb{E}\left[\left|z_k^{\prime}-\bar{z}'_k\right|^2\right] & =\mathbb{E}\left[|\mathbb{E}\left[\left.\left(y_k-\bar{y}_k\right) w_{k-1}\right| \mathcal{F}_{k-2} \right]|^2\right] \\
				& \leq \mathbb{E}\left[\mathbb{E}\left[|y_k-\bar{y}_k|^2 \mid \mathcal{F}_{k-2}\right] \cdot \mathbb{E}\left[w_{k-1}^2|\mathcal{F}_{k-2}\right]\right] \\
				& =\mathbb{E}\left[\left|y_k-\bar{y}_k\right|^2\right].
			\end{aligned}
		\end{equation}
		Therefore, by using the elementary inequality $|a+b+c|^2\le  3( |a|^2+|b|^2+|c|^2)$ and  the Lipschitz condition of $f$, we have
		\begin{equation} \label{eq:3.9}
			\begin{aligned}
				&{\mathbb{E}}\left[\left|y_{k}-\bar{y}_{k}\right|^2\right]\\  =&{\mathbb{E}}\left[\left|f\left(k+1, y'_{k+1}, z'_{k+1}\right)-\bar{f}\left(k+1, \bar{y}'_{k+1}, \bar{z}'_{k+1}\right)\right|^2\right] \\
				=&{\mathbb{E}}\left[\left|f\left(k+1, y'_{k+1}, z'_{k+1}\right)-f\left(k+1, \bar{y}'_{k+1}, \bar{z}'_{k+1}\right)+f\left(k+1, \bar{y}'_{k+1}, \bar{z}'_{k+1}\right)-\bar{f}\left(k+1, \bar{y}'_{k+1}, \bar{z}'_{k+1}\right)\right|^2\right] \\
				\leq& 3  {\mathbb{E}} \left\{ L^2_2\left|y'_{k+1}-\bar{y}'_{k+1}\right|^2
				+L^2_2\left|z'_{k+1}-\bar{z}'_{k+1}\right|^2+\left|f\left(k+1, \bar{y}'_{k+1}, \bar{z}'_{k+1}\right)-\bar{f}\left(k+1, \bar{y}'_{k+1}, \bar{z}'_{k+1}\right)\right|^2\right\} \\
				\leq&  6L^2_2  {\mathbb{E}} \left [\left|y_{k+1}-\bar{y}_{k+1}\right|^2\right]
				+ 3 {\mathbb{E}} \left [\left|f\left(k+1, \bar{y}'_{k+1}, \bar{z}'_{k+1}\right)-\bar{f}\left(k+1, \bar{y}'_{k+1}, \bar{z}'_{k+1}\right)\right|^2\right],
			\end{aligned}
		\end{equation}
		where \eqref{eq:3.7} and \eqref{eq:3.8} have been used in the last inequality.
		So  we have
		\begin{equation} \label{eq:3.10}
			\begin{aligned}
				&\sum_{k=0}^{\infty} e^{-\rho k}\mathbb{E} \left[\left|y_{k}-\bar{y}_{k}\right|^2\right]\\
				\leq&  6L^2_2 \sum_{k=0}^{\infty} e^{-\rho k}\ {\mathbb{E}} \left [\left|y_{k+1}-\bar{y}_{k+1}\right|^2\right]
				+  3\sum_{k=0}^{\infty}e^{-\rho k} {\mathbb{E}} \left [\left|f\left(k+1, \bar{y}'_{k+1}, \bar{z}'_{k+1}\right)-\bar{f}\left(k+1, \bar{y}'_{k+1}, \bar{z}'_{k+1}\right)\right|^2\right].
			\end{aligned}
		\end{equation}
		Let \(A_2 = \sum_{k=0}^{\infty} e^{-\rho k}\mathbb{E} \left[\left|y_{k}-\bar{y}_{k}\right|^2\right]\). Then  we can get that
		\[
		A_2 \le  6L^2_2 e^{\rho} (A_2 - e^{-\rho} \mathbb{E} [ | y_0-\bar{y}_{0} |^2 ])+ 3 e^{\rho}\sum_{k=0}^{\infty}e^{-\rho (k+1)} {\mathbb{E}} \left [\left|f\left(k+1, \bar{y}'_{k+1}, \bar{z}'_{k+1}\right)-\bar{f}\left(k+1, \bar{y}'_{k+1}, \bar{z}'_{k+1}\right)\right|^2\right].
		\]
		It is easy to see that
		\[
		A_2 (1 - 6L^2_2 e^{\rho}) \le   3 e^{\rho} \sum_{k=0}^{\infty}e^{-\rho (k+1)} {\mathbb{E}} \left [\left|f\left(k+1, \bar{y}'_{k+1}, \bar{z}'_{k+1}\right)-\bar{f}\left(k+1, \bar{y}'_{k+1}, \bar{z}'_{k+1}\right)\right|^2\right].
		\]
		Since \(\rho < -\ln(6L_{2}^2)\), \(1 - 6L_{2}^2 e^{\rho} > 0\). Therefore,
		\[
		A_2 \le \frac{3 e^{\rho}}{1 - 6L^2 _2e^{\rho}} \sum_{k=0}^{\infty}e^{-\rho (k+1)} {\mathbb{E}} \left [\left|f\left(k+1, \bar{y}'_{k+1}, \bar{z}'_{k+1}\right)-\bar{f}\left(k+1, \bar{y}'_{k+1}, \bar{z}'_{k+1}\right)\right|^2\right].
		\]
		Letting \(\mathcal{C} = \frac{3 e^{\rho}}{1 - 6L^2_2 e^{\rho}}\) (which depends only on \(L_2\) and $\rho$), we obtain \eqref{eq:3.4}. In particular, if we select $\bar{f}=0$, from  \eqref{eq:3.4},  we can obtain  \eqref{eq:3.3}. This proof is complete.
	\end{proof}

	{\color{blue}
		\begin{thm}\label{thm:2.2}
			Under Assumption \ref{ass:3.1}, if the constant \(\rho\) satisfies \(\rho < -\ln(6L_2^2)\), then the BS$\triangle$E $(f)$ admits a unique solution \(y(\cdot) \in L_{\mathbb{F}}^{2,\rho}(\mathbb{T}_{\infty};\mathbb{R}^n)\).
		\end{thm}
		
		\begin{proof}
			We first prove uniqueness. Suppose \(y(\cdot), \bar{y}(\cdot) \in L_{\mathbb{F}}^{2,\rho}(\mathbb{T}_{\infty};\mathbb{R}^n)\) are two solutions to BS$\triangle$E $(f)$. By Lemma \ref{lem:2.2} and estimate \eqref{eq:3.4}, setting \(\bar{f} = f\), we obtain
			\[
			\sum_{k=0}^\infty e^{-\rho k} \mathbb{E}\big[ |y_k - \bar{y}_k|^2 \big] \leq 0,
			\]
			which implies that for each $k$, \(y_k \equiv \bar{y}_k\) holds almost surely. Hence, the solution is unique.
			
			To prove existence, we use a finite-horizon approximation argument. First, we rewrite BS$\triangle$E $(f)$ as
			\[
			y_k = F(k+1, y'_{k+1}, z'_{k+1}) + \varphi_{k+1},
			\]
			where \(F(k, y, z) = f(k, y, z) - f(k, 0, 0)\), and \(\varphi_k = f(k, 0, 0)\) satisfies \(\sum_{k=0}^\infty e^{-\rho k} \mathbb{E}[|\varphi_k|^2] < \infty\) by assumption.
			
			For each integer \(n \geq 1\), define a truncated version of \(\varphi\) by
			\[
			\varphi^{(n)}_{k+1} = \varphi_{k+1} \cdot \mathbbm{1}_{\{k+1 \leq n\}} \quad k \in \mathbb{T}_{\infty}.
			\]
			Since \(\varphi(\cdot) \in L_{\mathbb{F}}^{2,\rho}(\mathbb{T}_{\infty};\mathbb{R}^n)\), it follows that
			\[
			\lim_{n \to \infty} \|\varphi^{(n)} - \varphi\|_\rho^2 =\lim_{n \to \infty} \sum_{k=n}^\infty e^{-\rho k} \mathbb{E}[|\varphi_{k+1}|^2] = 0.
			\]
			Next, consider the following BS$\triangle$E on the finite time horizon $\mathbb{T}_{N}:$
			\[
			\begin{cases}
				\bar{y}_k^{(n)} = F(k+1, \bar{y}_{k+1}^{(n)'}, \bar{z}_{k+1}^{(n)'}) + \varphi^{(n)}_{k+1}, & k\in\mathbb{T}_{n-1}, \\
				\bar{y}_n^{(n)} = 0,
			\end{cases}
			\]
			where
			\begin{align*}
				\bar{y}_{k + 1}^{(n)'} &= \mathbb{E}\left[\bar{y}_{k + 1}^{(n)} \mid \mathcal{F}_{k - 1}\right], \\
				\bar{z}_{k + 1}^{(n)'} &= \mathbb{E}\left[\bar{y}_{k + 1}^{(n)} \omega_k \mid \mathcal{F}_{k - 1}\right].
			\end{align*}
			According to the result in \cite{27}, this finite time horizon equation admits a unique adapted solution \(\{\bar{y}_k^{(n)}\}_{k=0}^n\).
			
			We now extend \(\bar{y}^{(n)}(\cdot)\) to an infinite sequence \(y^{(n)}(\cdot) \in L_{\mathbb{F}}^{2,\rho}(\mathbb{T}_{\infty};\mathbb{R}^n)\) by defining
			\[
			y^{(n)}_k =
			\begin{cases}
				\bar{y}_k^{(n)}, & k\in\mathbb{T}_{n}, \\
				0, & k\in\{n+1,n+2,\cdots\}.
			\end{cases}
			\]
			Then \(y^{(n)}(\cdot)\) satisfies
			\[
			y^{(n)}_k = F(k+1, y^{(n)'}_{k+1}, z^{(n)'}_{k+1}) + \varphi^{(n)}_{k+1}, \quad \forall k\in \mathbb{T}_{\infty}.
			\]
			Let \(m > n\), and define \(\delta y^{(m,n)}(\cdot) := y^{(m)}(\cdot) - y^{(n)}(\cdot)\). By Lemma \ref{lem:2.2}, we have
			\[
			\| \delta y^{(m,n)}(\cdot) \|_\rho^2 \leq C \| \varphi^{(m)}(\cdot) - \varphi^{(n)}(\cdot) \|_\rho^2.
			\]
			Combined with the convergence \(\| \varphi^{(m)}(\cdot) - \varphi^{(n)}(\cdot) \|_\rho \to 0\), we conclude that \(\{y^{(n)}(\cdot)\}\) is a Cauchy sequence in the Hilbert space \(L_{\mathbb{F}}^{2,\rho}(\mathbb{T}_{\infty};\mathbb{R}^n)\). Therefore, there exists a unique limit \(y(\cdot) \in L_{\mathbb{F}}^{2,\rho}(\mathbb{T}_{\infty};\mathbb{R}^n)\) such that
			\[
			\lim_{n \to \infty} \| y^{(n)}(\cdot) - y(\cdot) \|_\rho = 0.
			\]
			It remains to verify that the limit \(y(\cdot)\) solves BS$\triangle$E $(f)$. For each fixed \(k \geq 0\), we have
			\[
			e^{-\rho k} \mathbb{E}[|y^{(n)}_k - y_k|^2] \leq \| y^{(n)}(\cdot) - y(\cdot) \|_\rho^2 \to 0,
			\]
			so \(y^{(n)}_k \to y_k\) in \(L^2(\Omega)\). It then follows from the continuity of conditional expectation that
			\[
			\mathbb{E}[|y^{(n)'}_{k+1} - y'_{k+1}|^2] \to 0, \quad
			\mathbb{E}[|z^{(n)'}_{k+1} - z'_{k+1}|^2] \to 0.
			\]
			By the  Lipschitz continuity of \(F\), we obtain
			\[
			\mathbb{E}\left[ \left| F(k+1, y^{(n)'}_{k+1}, z^{(n)'}_{k+1}) - F(k+1, y'_{k+1}, z'_{k+1}) \right|^2 \right] \to 0.
			\]
			Together with the convergence \(\varphi^{(n)}(\cdot) \to \varphi(\cdot)\) in \(L_{\mathbb{F}}^{2,\rho}(\mathbb{T}_{\infty};\mathbb{R}^n)\), we conclude that
			\[
			y_k = F(k+1, y'_{k+1}, z'_{k+1}) + \varphi_{k+1}, \quad \mathbb{P}\text{-a.s.}, \ \forall k\in \mathbb{T}_{\infty}.
			\]
			Finally, we note that for any finite integer \(N \geq 0\), \(y^{(n)}_N = \bar{y}^{(n)}_N\) is \(\mathcal{F}_{N-1}\)-measurable for \(n > N\), and the limit \(y_N = \lim_{n \to \infty} y^{(n)}_N\) remains \(\mathcal{F}_{N-1}\)-measurable. Hence, the measurability of the limit process is preserved. This completes the proof.
	\end{proof}}

	\section{FBS$\bigtriangleup $Es with Domination-Monotonicity Conditions}\label{sec:3}
	Drawing on the preliminary results from the preceding section, our focus in this section will be on the investigation of the  FBS$\bigtriangleup $E \eqref{eq:1.1}.
	For the coefficients $(\Lambda,\Gamma)$ of FBS$\bigtriangleup $E \eqref{eq:1.1}, we now make the following assumptions.
	
	\begin{ass}\label{ass:4.1}
		(i) For any $y\in \mathbb{R}^n$,  $\Lambda (y)$ is deterministic. Furthermore, for any
		$\theta=(x,y',z')\in \mathbb{R}^{3n}$, $
		\Gamma (k,\theta )$ is $\mathcal{F}_{k-1}$-measurable. Moreover, $(\Lambda(0),\Gamma (\cdot ,0,0,0) )\in \mathcal{H}^{2,\rho}(\mathbb{T}_{\infty})$ for some 
$\rho\in \mathbb R$;
		(ii) The mapping $\Gamma$ is uniformly Lipschitz continuous, i.e., for any $x,\bar{x}, y,\bar{y},y',\bar{y}',z',\bar{z}'\in \mathbb{R}^n$, there exists eight constants $L, L_{1}, L_{2},L_3,L_4,L_5,L_6,L_7>0$ such that
		\begin{equation}
			\left\{\begin{aligned}
				&|\Lambda(y)-\Lambda(\bar{y})|\le L|y-\bar{y}|,\\
				&|b(k,x,y',z' )-b(k,\bar{x},\bar{y}',\bar{z}' )|\le L_{1}|x -\bar{x}|+L_3|y' -\bar{y}'|+L_4|z' -\bar{z}'|,\\
				&|\sigma(k,x,y',z' )-\sigma(k,\bar{x},\bar{y}',\bar{z}' )|\le L_{1}|x -\bar{x}|+L_5|y' -\bar{y}'|+L_6|z' -\bar{z}'|,\\
				&|f(k+1,x,y',z')-f(k+1,\bar{x},\bar{y}',\bar{z}')|\le L_{7}|x-\bar{x}|+L_2\big(|y'-\bar{y}' |+|z'-\bar{z}' |\big).
			\end{aligned}\right.\nonumber
		\end{equation}
	\end{ass}
	Apart from the aforementioned Assumption \ref{ass:4.1}, we present the following domination-monotonicity conditions for the coefficients $(\Lambda,\Gamma)$ to make future studies more convenient.
	\begin{ass}\label{ass:4.2}
		There exist two constants $\mu\ge 0$, $\nu\ge0$,   a matrix $M \in \mathbb{R}^{{m} \times n}$, along with a series of matrix-valued processes $A(\cdot), B(\cdot),C(\cdot)\in  L_{\mathbb{F}}^\infty\left(\mathbb{T}_{\infty};\mathbb{R}^{m\times n}\right)$,
		such that the subsequent conditions are satisfied.:
		
		(i) One of the following two cases is actually true. Case $1$: $\mu>0$ and $\nu=0$. Case $2$: $\mu=0$ and $\nu>0$.
		
		(ii) (domination condition) For  all $k\in \mathbb{T}_{\infty} $, almost all $\omega \in \Omega$ and any $x,\bar{x}, y,\bar{y},y',\bar{y}',z',\bar{z}'\in \mathbb{R}^n$,
		\begin{equation}\label{eq:4.2}
			\left\{\begin{aligned}
				&|\Lambda(y)-\Lambda(\bar{y})|\le \frac{1}{\mu }|M\widehat{y}|,\\
				& |f(k+1,x,y',z')-f(k+1,\bar{x},y',z')| \le \frac{1}{\nu}\big|A_k\widehat{x}\big|,\\
				& |h(k,x,y',z')-h(k,x,\bar{y}',\bar{z}')| \le \frac{1}{\mu}\bigg|B_k\widehat{y}'+C_k\widehat{z}'\bigg|,
			\end{aligned}\right.
		\end{equation}
		where $h=b, \sigma$, and $\widehat{x}=x-\bar{x}$, $\widehat{y}=y-\bar{y}$,  $\widehat{y}'=y'-\bar{y}'$, $\widehat{z}'=z'-\bar{z}'$.
		
		It is worth noting that there is a minor issue of notation abuse in the aforementioned conditions. Specifically, when $\mu=0$ (resp. $\nu=0$), $\frac{1}{\mu} $ (resp. $\frac{1}{\nu} $) is used to signify $+\infty$. Put differently, if $\mu=0$ or $\nu=0$, the associated domination limitations will disappear.
		
		(iii) (monotonicity condition) For  all $k\in \mathbb{T}_{\infty} $, almost all $\omega \in \Omega$ and any $\theta=(x,y',z'),\bar{\theta}=(\bar{x},\bar{y}',\bar{z}') \in \mathbb{R}^{3n}$, $y,\bar{y}\in \mathbb{R}^n$,
		\begin{equation}\label{eq:4.3}
			\left\{\begin{aligned}
				&\langle\Lambda(y)-\Lambda(\bar{y}), \widehat{y}\rangle \le -\mu|M \widehat{y}|^2,\\
				& \langle\varUpsilon (k,\theta)-\varUpsilon (k,\bar{ \theta}),\widehat{\theta}\rangle
				\leq -\nu\big|A_k\widehat{x}\big|^2-\mu \bigg|B_k\widehat{y}'+C_k\widehat{z}'\bigg|^2,
			\end{aligned}\right.
		\end{equation}
		where $\varUpsilon (k,\theta) = f(k+1,\theta) +e^{-\rho} b(k,\theta)+e^{-\rho} \sigma(k,\theta)$
		and $\widehat{\theta } =\theta - \bar{\theta} $.
	\end{ass}
	\begin{rmk}\label{rmk:4.1}
		(i) In Assumption \ref{ass:4.2}-(ii), the constant $1/\mu$ and $1/\nu$ could be substituted by $\mathbb{K}/\mu$ and $\mathbb{K}/\nu  (\mathbb{K}>0)$. However, for the purpose of simplicity, we choose to leave out the constant $\mathbb{K}$;\\
		(ii) A symmetric version of Assumption \ref{ass:4.2}-(iii) exists and is presented as follows:
		
		For  all $k\in \mathbb{T}_{\infty} $, almost all $\omega \in \Omega$ and any $\theta,\bar{\theta} \in \mathbb{R}^{3n}$,
		\begin{equation}\label{eq:4.4}
			\left\{\begin{aligned}
				&\langle\Lambda(y)-\Lambda(\bar{y}), \widehat{y}\rangle \ge \mu|M \widehat{y}|^2,\\
				&\langle\varUpsilon (k,\theta)-\varUpsilon (k,\bar{ \theta}),\widehat{\theta}\rangle
				\ge \nu\big|A_k\widehat{x}\big|^2+\mu \bigg|B_k\widehat{y}'+C_k\widehat{z}'\bigg|^2.
			\end{aligned}\right.
		\end{equation}
		The symmetry  between (\ref{eq:4.3}) and (\ref{eq:4.4}) is readily apparent. Therefore we omit the detailed proofs. Similar proofs can be found by referring to Yu \cite{25}.

	\end{rmk}
	For convenience, we will use the following notations:
	\begin{equation}\label{eq:4.5}
		\left\{\begin{aligned}
			&P(k,x)=A_kx,\\
			&P(k,\widehat{x})=A_k\widehat{x},\\
			&Q(k,y',z')=B_ky'+C_kz',\\
			&Q(k,\widehat{y}',\widehat{z}')=B_k\widehat{y}'+C_k\widehat{z}' ,\quad k \in\mathbb{T}_{\infty}.
		\end{aligned}\right.
	\end{equation}
	We are now in a position to present the principal theorem of this section.
	\begin{thm}\label{thm:4.1}
Suppose the coefficients $(\Lambda ,\Gamma)$ of FBS$\bigtriangleup $E $(\Lambda ,\Gamma)$satisfy Assumptions \ref{ass:4.1}-\ref{ass:4.2}. Furthermore, assume that $ \ln(4L^2_1)< -\ln(6L^2_2)$ and
$ \ln(4L^2_1)<\rho< -\ln(6L^2_2)$.
		Then FBS$\bigtriangleup $E $(\Lambda ,\Gamma)$ admits a unique solution $(x(\cdot),y(\cdot) )\in N^{2,\rho}(\mathbb{T}_{\infty};\mathbb{R}^{2n})$. Moreover, we can get the following estimate:
		\begin{equation}\label{eq:4.6}
			\|x(\cdot)\|_{\rho}^2+\|y(\cdot)\|_{\rho}^2 \le K\mathbb{E}[\mathrm{I}],
		\end{equation}
		where
		\begin{eqnarray}\label{eq:4.7}
			\begin{aligned}
				\mathrm{I} = \sum_{k=0}^{\infty} e^{-\rho k} |b(k,0,0,0)|^2 + \sum_{k=0}^{\infty} e^{-\rho k} |\sigma (k,0,0,0)|^2 + \sum_{k=0}^{\infty} e^{-\rho (k+1)} |f(k+1,0,0,0)|^2 + |\Lambda (0)|^2,
			\end{aligned}
		\end{eqnarray}
		and \( K \) is a positive constant depending only on the Lipschitz constants, \( \mu \), \( \nu \), and the bounds of all \( A(\cdot) \), \( B(\cdot) \), \( C(\cdot) \), \(M\).
		What's more, assume that \( (\bar{\Lambda},\bar{ \Gamma}) \) is another set of coefficients which satisfies Assumptions \ref{ass:4.1}-\ref{ass:4.2}. Also, suppose that \( (\bar{x}(\cdot),\bar{y}(\cdot)) \in N^{2,\rho}(\mathbb{T}_{\infty};\mathbb{R}^{2n}) \) is a solution to FBS$\bigtriangleup $E \( (\bar{\Lambda},\bar{ \Gamma}) \). Then the following estimate holds:
		\begin{equation}\label{eq:4.8}
			\|\widehat{x}(\cdot)\|_{\rho}^2+\|\widehat{y}(\cdot)\|_{\rho}^2 \le K\mathbb{E}[\widehat{\mathrm{I}}].
		\end{equation}
		Here we denote $\widehat{x}(\cdot)=x(\cdot)-\bar{x}(\cdot) $,     $\widehat{y}(\cdot)=y(\cdot)-\bar{y}(\cdot) $ and
		\begin{eqnarray}\label{eq:4.9}
			\begin{aligned}
				\widehat{\mathrm{I}} = &  \sum_{k=0}^{\infty} e^{-\rho k} |b(k,\bar{\theta}_k) - \bar{b}(k,\bar{\theta}_k)|^2 + \sum_{k=0}^{\infty} e^{-\rho k} |\sigma(k,\bar{\theta}_k) - \bar{\sigma}(k,\bar{\theta}_k)|^2
				\\&+ \sum_{k=0}^{\infty} e^{-\rho (k+1)} |f(k+1,\bar{\theta}_k) - \bar{f}(k+1,\bar{\theta}_k)|^2 + |\Lambda (\bar{y}_0)-\bar{\Lambda}(\bar{y}_0)|^2,
			\end{aligned}
		\end{eqnarray}
		where \( \theta_k = (x_k, y_{k+1}', z_{k+1}') = (x_k, \mathbb{E} [y_{k+1}|\mathcal{F}_{k-1}], \mathbb{E} [y_{k+1}w_k|\mathcal{F}_{k-1}]) \) and \( \bar{\theta}_k = (\bar{x}_k, \bar{y}_{k+1}', \bar{z}_{k+1}') \), and the constant \( K \) is  identical to that in \eqref{eq:4.6}.
	\end{thm}
	\begin{rmk}\label{rmk:3.2}
		As a matter of fact, it is easy to see that under conditions $ \ln(4L^2_1)< -\ln(6L^2_2)$ and $ \ln(4L^2_1)<\rho <-\ln(6L^2_2)$,  Lemmas \ref{lem:2.1}-\ref{lem:2.2} and Theorems \ref{thm:2.1}-\ref{thm:2.2} can be applied latter.
	\end{rmk}
	Subsequently, we will focus our efforts on proving Theorem \ref{thm:4.1}. Given the symmetry of the monotonicity conditions \eqref{eq:4.3} and \eqref{eq:4.4}, we will only present the detailed proofs based on the monotonicity condition \eqref{eq:4.3}. Firstly, it is essential to derive an a priori estimate for the solution  FBS$\bigtriangleup $E \eqref{eq:1.1}.
	\begin{lem}\label{lem:4.2}
		Let the given coefficients \( (\Lambda,\Gamma) \) satisfy Assumptions \ref{ass:4.1}-\ref{ass:4.2}. Here we still assume that $ \ln(4L^2_1)< -\ln(6L^2_2)$ and  $ \ln(4L^2_1)<\rho <-\ln(6L^2_2)$. Assume that \( (x(\cdot),y(\cdot) ) \in N^{2,\rho}(\mathbb{T}_{\infty};\mathbb{R}^{2n}) \) is the solution of FBS$\bigtriangleup $E \( (\Lambda,\Gamma) \).  Furthermore, let \( (\bar{\Lambda} ,\bar{\Gamma}) \) be another set of coefficients satisfying Assumptions \ref{ass:4.1}-\ref{ass:4.2}  and suppose \( (\bar{x}(\cdot),\bar{y}(\cdot)) \in N^{2,\rho} (\mathbb{T}_{\infty};\mathbb{R}^{2n}) \) is the solution to FBS$\bigtriangleup $E \( (\bar{\Lambda} ,\bar{\Gamma}) \). Then the  estimate \eqref{eq:4.8} holds.
	\end{lem}
	
	\begin{proof}
		The positive constant $K$ could differ from line to line.
		Utilizing the estimate \eqref{eq:2.4} in Lemma \ref{lem:2.1}, we can get that
		\begin{equation}\label{eq:4.15}
			\begin{aligned}
				\|\widehat{x}(\cdot)\|_{\rho}^2
				\le K\mathbb{E}\Bigg\{
				&\big| \Lambda(y_0) - \bar{\Lambda}(\bar{y}_0)\big|^2+ \sum_{k=0}^{\infty}e^{-\rho k} \bigg| b(k, \bar{x}_k, y_{k+1}', z'_{k+1}) - \bar{b}(k, \bar{x}_k, \bar{y}_{k+1}', \bar{z}_{k+1}')  \bigg|^2  \\
				&+ \sum_{k=0}^{\infty}e^{-\rho k} \Bigg| \sigma(k, \bar{x}_k, y_{k+1}', z_{k+1}') - \bar{\sigma}(k, \bar{x}_k, \bar{y}_{k+1}', \bar{z}_{k+1}') \Bigg|^2
				\Bigg\}.
			\end{aligned}
		\end{equation}
		Likewise, by applying the estimate \eqref{eq:3.4} in Lemma \ref{lem:2.2},  we can get
		
		\begin{equation}\label{eq:4.16}
			\begin{aligned}
				\|\widehat{y}(\cdot)\|_{\rho}^2
				\le & K\mathbb{E}\Bigg\{ \sum_{k=0}^{\infty}e^{-\rho (k+1)} \bigg| f(k+1, x_k, \bar{y}_{k+1}', \bar{z}_{k+1}') - \bar{f}(k+1, \bar{x}_k, \bar{y}_{k+1}', \bar{z}_{k+1}')\bigg|^2  \Bigg\}.
			\end{aligned}
		\end{equation}
		We take into account the equation
		\begin{equation} \label{eq:4.17}
			\begin{aligned}
				\mathbb{E}\left[e^{-\rho N}\left\langle\widehat{x}_N, \widehat{y}_N\right\rangle - e^{-\rho 0}\left\langle\widehat{x}_0, \widehat{y}_0\right\rangle\right]
				= & \sum_{k=0}^{N-1} \mathbb{E}\bigg[e^{-\rho (k+1)}\left\langle \widehat{x}_{k+1}, \widehat{y}_{k+1}\right\rangle - e^{-\rho k}\left\langle \widehat{x}_{k}, \widehat{y}_{k}\right\rangle\bigg],
			\end{aligned}
		\end{equation}
		where the left side is equal to
		
		\begin{equation} \label{eq:4.18}
			\begin{aligned}
				\mathbb{E}\left[e^{-\rho N}\left\langle\widehat{x}_N, \widehat{y}_N\right\rangle - \left\langle\widehat{x}_0, \widehat{y}_0\right\rangle\right]
				=  \mathbb{E}\left[e^{-\rho N}\left\langle\widehat{x}_N,\widehat{y}_N \right\rangle - \left\langle\Lambda\left(y_0\right) - \bar{\Lambda}\left(\bar{y}_0\right), \widehat{y}_0\right\rangle \right]
			\end{aligned}
		\end{equation}
		and the right side is equal to
		\begin{equation} \label{eq:4.19}
			\begin{aligned}
				&\displaystyle \sum_{k=0}^{N-1}\mathbb E\bigg[e^{-\rho(k+1)}\left \langle \widehat{x}_{k+1},\widehat{y}_{k+1}
				\right \rangle-e^{-\rho k}\left \langle \widehat{x}_{k},\widehat{y}_{k} \right \rangle\bigg]
				\\=& \sum_{k=0}^{N-1} \mathbb{E}\bigg[\mathbb{E}\left[e^{-\rho(k+1)}\left\langle\widehat{x}_{k+1}, \widehat{y}_{k+1}\right\rangle \mid \mathcal{F}_{k-1}\right]-e^{-\rho k}\left\langle\widehat{x}_k, \widehat{y}_k\right\rangle\bigg]
				\\=& \sum_{k=0}^{N-1} \mathbb{E}\bigg[\mathbb{E}\left[e^{-\rho (k+1)} \left\langle\widehat{b}_k+\widehat{\sigma}_k\omega_k, \widehat{y}_{k+1}\right\rangle  \mid \mathcal{F}_{k-1}\right]-e^{-\rho k}\left\langle\widehat{x}_k, \widehat{y}_k\right\rangle\bigg]
				\\ =& \sum_{k=0}^{N-1} \mathbb{E}\left[e^{-\rho (k+1)}\left( \left\langle\widehat{b}_k, \widehat{y}_{k+1}^{\prime}\right\rangle+\left\langle\widehat{\sigma}_k, \widehat{z}_{k+1}^{\prime}\right\rangle\right) +e^{-\rho k} \left\langle\widehat{f}_{k+1}, \widehat{x}_k\right\rangle \right]\\
				=&\sum_{k=0}^{N-1}\mathbb{E}\left[e^{-\rho k}\left \langle \varUpsilon (k,\theta_k )-\bar{\varUpsilon}(k,\bar{\theta}_k ) ,\widehat{\theta}_k \right \rangle \right],
			\end{aligned}
		\end{equation}
		where $\widehat{b}_k=b(k,\theta_k) - \bar{b}(k,\bar{\theta}_k)$, and similarly for $\widehat{\sigma}, \widehat{f}$, etc.
		Putting \eqref{eq:4.18} and \eqref{eq:4.19} into \eqref{eq:4.17}, we get that
		\begin{equation}\label{eq:4.20}
			\begin{aligned}
				\mathbb{E}\left[e^{-\rho N}\left\langle\widehat{x}_N,\widehat{y}_N \right\rangle - \left\langle\Lambda\left(y_0\right) - \bar{\Lambda}\left(\bar{y}_0\right), \widehat{y}_0\right\rangle \right] =\sum_{k=0}^{N-1}\mathbb{E}\left[e^{-\rho k}\left \langle \varUpsilon (k,\theta_k )-\bar{\varUpsilon}(k,\bar{\theta}_k ) ,\widehat{\theta}_k \right \rangle \right].
			\end{aligned}
		\end{equation}
		It is easy to verify that
		\begin{equation*}
			\lim_{N \to \infty}e^{-\rho N}|\left\langle\widehat{x}_N,\widehat{y}_N \right\rangle|\leq \lim_{N \to \infty}e^{-\rho N}|\widehat{x}_N||\widehat{y}_N|\leq \lim_{N \to \infty}e^{-\rho N}\left(|\widehat{x}_N|^{2}+|\widehat{y}_N|^{2}\right)=0.
		\end{equation*}
		Therefore,
		\begin{equation*}
			\lim_{N \to \infty}e^{-\rho N}\left\langle\widehat{x}_N,\widehat{y}_N \right\rangle=0.
		\end{equation*}
		Let $N\rightarrow\infty$ on both sides of \eqref{eq:4.20}. We can get that
		\begin{equation}\label{eq:4.21}
			\begin{aligned}
				\mathbb{E}\left\lbrace \sum_{k=0}^{\infty}e^{-\rho k}\left \langle \varUpsilon (k,\theta_k )-\bar{\varUpsilon}(k,\bar{\theta}_k ) ,\widehat{\theta}_k \right \rangle +\left\langle\Lambda\left(y_0\right) - \bar{\Lambda}\left(\bar{y}_0\right), \widehat{y}_0\right\rangle\right\rbrace =0 .
			\end{aligned}
		\end{equation}
		Consequently, in conjunction with the monotonicity conditions in Assumption \ref{ass:4.2}-(iii), \eqref{eq:4.21} turns into
		\begin{equation}\label{eq:4.22}
			\begin{aligned}
				&\quad\mathbb{E}\Bigg \{ \mu |M\widehat{y}_0|^2+\nu\sum_{k=0}^{\infty}e^{-\rho k}\big|P(k,\widehat{x}_k)\big|^2
				+\mu \sum_{k=0}^{\infty}e^{-\rho k}\big |Q(k,\widehat y_{k+1}',\widehat z_{k+1}')\big |^2\Bigg \}\\
				&\le \mathbb{E}\Bigg \{\sum_{k=0}^{\infty}e^{-\rho k}\left \langle \varUpsilon (k,\bar{\theta}_k )-\bar{\varUpsilon}(k,\bar{\theta}_k ) ,\widehat{\theta}_k\right\rangle+\left\langle\Lambda\left(\bar{y}_0\right) - \bar{\Lambda}\left(\bar{y}_0\right), \widehat{y}_0\right\rangle \Bigg \}.
			\end{aligned}
		\end{equation}
		The subsequent proofs will be split into two cases in accordance with Assumption \ref{ass:4.2}-(i).\\
		\textbf{Case 1}: $\mu>0$ and $\nu=0$. By making use of the domination conditions \eqref{eq:4.2} in Assumption \ref{ass:4.2}-(ii) for the estimate \eqref{eq:4.15}, we obtain
		\begin{equation}\label{eq:4.23}
			\begin{aligned}
				\|\widehat{x}(\cdot)\|_{\rho}^2\le &K\mathbb E\Bigg \{|\Lambda\left(\bar{y}_0\right) - \bar{\Lambda}\left(\bar{y}_0\right) |^2+|M\widehat{y}_0|^2+\sum_{k=0}^{\infty}e^{-\rho k}|b(k, \bar{\theta}_k) - \bar{b}(k, \bar{\theta}_k) |^2\\
				&+\sum_{k=0}^{\infty}e^{-\rho k}|\sigma(k, \bar{\theta}_k) - \bar{\sigma}(k, \bar{\theta}_k) |^2+\sum_{k=0}^{\infty}e^{-\rho k}\big |Q(k,\widehat y_{k+1}',\widehat z_{k+1}')\big |^2\Bigg \}.
			\end{aligned}
		\end{equation}
		By applying the Lipschitz condition to the estimate given in \eqref{eq:4.16}, we have
		\begin{equation}\label{eq:4.24}
			\begin{aligned}
				\|\widehat{y}(\cdot)\|_{\rho}^2\le K\mathbb{E}\bigg \{\sum_{k=0}^{\infty}e^{-\rho (k+1)}|f(k+1, \bar{\theta}_k) - \bar{f}(k+1, \bar{\theta}_k)|^2
				+\displaystyle \sum_{k=0}^{\infty} e^{-\rho (k+1)}|\widehat{x}_k |^2\bigg\}.
			\end{aligned}
		\end{equation}
		Consequently, the combination of  \eqref{eq:4.23} and \eqref{eq:4.24} results in the following inequality:
		\begin{equation}\label{eq:4.25}
			\|\widehat{x}(\cdot)\|_{\rho}^2+\|\widehat{y}(\cdot)\|_{\rho}^2\le K\mathbb{E}\Bigg \{\widehat{\mathrm{I}}+\sum_{k=0}^{\infty}e^{-\rho k}\big |Q(k,\widehat y_{k+1}',\widehat z_{k+1}')\big |^2+|M\widehat{y_0}|^2\Bigg \},	
		\end{equation}
		where $\widehat{\mathrm{I}}$ is defined by \eqref{eq:4.9}. Finally, by combining \eqref{eq:4.22} and \eqref{eq:4.25} and using the inequality $ab\le\frac{1}{4\varepsilon}a^2+\varepsilon b^2$, we can derive the following inequality:
		\begin{equation}\label{eq:4.26}
			\begin{aligned}
				&\quad	\|\widehat{x}(\cdot)\|_{\rho}^2+\|\widehat{y}(\cdot)\|_{\rho}^2\\
				&\le K\mathbb{E}\Bigg \{\widehat{\mathrm{I}}
				+\sum_{k=0}^{\infty}e^{-\rho k}\left \langle \varUpsilon (k,\bar{\theta}_k )-\bar{\varUpsilon}(k,\bar{\theta}_k ) ,\widehat{\theta}_k\right\rangle+\left\langle\Lambda\left(\bar{y}_0\right) - \bar{\Lambda}\left(\bar{y}_0\right), \widehat{y}_0\right\rangle\Bigg \}
				\\&\le K\mathbb{E}\Bigg \{\widehat{\mathrm{I}}+2\varepsilon\left [ \|\widehat{x}(\cdot)\|_{\rho}^2+\|\widehat{y}(\cdot)\|_{\rho}^2\right ]\Bigg\}.
			\end{aligned}
		\end{equation}
		By taking $\varepsilon$ small enough such that $2K\varepsilon<1$, the desired estimate \eqref{eq:4.8} can be readily obtained. This completes the proof in this case.\\
		\textbf{Case 2}: $\mu=0$ and $\nu>0$. In a different way, we make use of the Lipschitz conditions for the estimate \eqref{eq:4.15} so as to obtain
		\begin{equation}\label{eq:4.27}
			\begin{aligned}
				\|\widehat{x}(\cdot)\|_{\rho}^2\le& K\mathbb{E}\Bigg \{\displaystyle|\Lambda\left(\bar{y}_0\right) - \bar{\Lambda}\left(\bar{y}_0\right)|^2
				+ \sum_{k=0}^{\infty}e^{-\rho k} |\widehat{y}_k|^2+\sum_{k=0}^{\infty}e^{-\rho k}|b(k, \bar{\theta}_k) - \bar{b}(k, \bar{\theta}_k) |^2\\
				&+\sum_{k=0}^{\infty}e^{-\rho k}|\sigma(k, \bar{\theta}_k) - \bar{\sigma}(k, \bar{\theta}_k) |^2\Bigg \}.
			\end{aligned}
		\end{equation}
		Based on the domination conditions \eqref{eq:4.2} in Assumption \ref{ass:4.2}-(ii), we infer from \eqref{eq:4.16} and get
		\begin{equation}\label{eq:4.28}
			\begin{aligned}
				&\quad\|\widehat{y}(\cdot)\|_{\rho}^2
				\le K\mathbb{E}\bigg \{ \sum_{k=0}^{\infty}e^{-\rho k} |P(k,\widehat{x}_k)|^2+\sum_{k=0}^{\infty}e^{-\rho (k+1)}|f(k+1, \bar{\theta}_k) - \bar{f}(k+1, \bar{\theta}_k) |^2
				\Bigg \}.
			\end{aligned}
		\end{equation}
		Consequently, when we combine \eqref{eq:4.27} and \eqref{eq:4.28}, we can get that
		\begin{equation}\label{eq:4.29}
			\begin{aligned}
				\|\widehat{x}(\cdot)\|_{\rho}^2+\|\widehat{y}(\cdot)\|_{\rho}^2
				\le K\mathbb{E}\Bigg \{\widehat{\mathrm{I} }+\sum_{k=0}^{\infty}e^{-\rho k}|P(k,\widehat{x}_k)|^2\Bigg \},
			\end{aligned}
		\end{equation}
		where $\widehat{\mathrm{I}}$ is defined by \eqref{eq:4.8}. Ultimately, \eqref{eq:4.22} and \eqref{eq:4.29} work in conjunction to derive
		\begin{equation}
			\begin{aligned}
				\quad\|\widehat{x}(\cdot)\|_{\rho}^2+\|\widehat{y}(\cdot)\|_{\rho}^2\le K\mathbb{E}\Bigg \{\widehat{\mathrm{I}}
				+\sum_{k=0}^{\infty}e^{-\rho k}\left \langle \varUpsilon (k,\bar{\theta}_k )-\bar{\varUpsilon}(k,\bar{\theta}_k ) ,\widehat{\theta}_k\right\rangle+\left\langle\Lambda\left(\bar{y}_0\right) - \bar{\Lambda}\left(\bar{y}_0\right), \widehat{y}_0\right\rangle\Bigg \}.
			\end{aligned}
		\end{equation}
		The rest of the proof is the same as \eqref{eq:4.26} in Case 1 and thus the proof in this case is finished. As a result, the entire proof of the lemma has been accomplished.
	\end{proof}
	Given any $\big(\xi,\phi(\cdot) \big)\in \mathcal{H}^{2,\rho}(\mathbb{T}_{\infty})$ where $\phi(\cdot) =(\varphi(\cdot) ,\psi(\cdot) ,\gamma(\cdot))$, we proceed to introduce  a family of FBS$\bigtriangleup $Es parameterized by $\alpha\in[0,1]$ as follows:
	\begin{equation}\label{eq:4.10}
		\left\{\begin{aligned}
			x^{\alpha}_{k+1} =&\big[b^{\alpha}(k,\theta_{k}^{\alpha}) +\psi_k\big]+\big[\sigma^{\alpha}(k,\theta_{k}^{\alpha}) +\gamma_k\big] \omega_k,    \\
			y_k^{\alpha}=& -\big[f^{\alpha}(k+1,\theta_k^{\alpha})+\varphi_k\big], \quad k \in\mathbb{T}_{\infty},\\
			x_0^{\alpha} =&\Lambda^\alpha(y_0^{\alpha})+\xi,\\
		\end{aligned}\right.
	\end{equation}
	where $\theta ^\alpha _k={(x^\alpha_k},{y'^\alpha_{k+1}},{z'^\alpha_{k+1}})=(x^{\alpha }_k,\mathbb{E} [y^{\alpha}_{k+1}|\mathcal{F}_{k-1}]
	,\mathbb{E} [y^{\alpha}_{k+1} \omega_k|\mathcal{F}_{k-1}]
	).$ At this point, we have utilized the following notations: for any $(k,\omega,y,\theta)\in\mathbb{T}_{\infty}\times\Omega\times \mathbb{R}^{n} \times \mathbb{R}^{3n}$,\\
	\begin{equation}\label{eq:4.11}
		\left\{\begin{aligned}
			&\Lambda^{\alpha}(y)=\alpha\Lambda(y)-(1-\alpha)\mu M^\top My,\\
			& f^\alpha (k+1,\theta)=\alpha f(k+1,\theta )-(1-\alpha)\nu\bigg[A_k^{\top}P(k,x)\bigg],\\
			& b^\alpha (k,\theta)=\alpha b(k,\theta)-(1-\alpha)\mu\bigg[ B_k^\top Q(k,y',z')\bigg],\\
			& \sigma ^\alpha (k,\theta)=\alpha \sigma (k,\theta )-(1-\alpha)\mu \bigg[C_k^\top Q(k,y',z') \bigg],\\
		\end{aligned}\right.
	\end{equation}
	where $P(k,x)$ ,$Q(k,y',z')$ are defined by \eqref{eq:4.5}.
	We denote $\Gamma ^\alpha (k,\theta ):=(f^\alpha(k+1,\theta ) ,b^\alpha(k,\theta ) ,\sigma ^\alpha
	(k,\theta ))$.
	For the sake of generality, we  can chose  the Lipschitz constants $L,\ L_3,\ L_4,\ L_5,\ L_6,\ L_7$ of the
	coefficients $(\Lambda,\Gamma)$ which are larger than
	\begin{equation}
		\max \{\mu, \nu\}\left(|M|+\|A(\cdot)\|_{L_{\mathbb{F}}^{\infty}(\mathbb{T}_{\infty}; \mathbb{R}^{m\times n})}+\|B(\cdot)\|_{L_{\mathbb{F}}^{\infty}(\mathbb{T}_{\infty}; \mathbb{R}^{m\times n})}+\|C(\cdot)\|_{L_{\mathbb{F}}^{\infty}(\mathbb{T}_{\infty}; \mathbb{R}^{m\times n})}\right)^2\nonumber
	\end{equation}
	and the constants $\mu$ and $\nu$ in Assumption \ref{ass:4.2}-(i) satisfy the following condition:
	\begin{equation}
		\begin{aligned}
			(\frac{1}{\mu })^2,(\frac{1}{\nu})^2\ge \text{max}\Big\{|M|,\|A(\cdot)\|_{L_{\mathbb{F}}^{\infty}(\mathbb{T}_{\infty}; \mathbb{R}^{m\times n})},\|B(\cdot)\|_{L_{\mathbb{F}}^{\infty}(\mathbb{T}_{\infty}; \mathbb{R}^{m\times n})},\|C(\cdot)\|_{L_{\mathbb{F}}^{\infty}(\mathbb{T}_{\infty}; \mathbb{R}^{m\times n})}\Big\}.\nonumber
		\end{aligned}
	\end{equation}
	Then, it can be easily confirmed that for any $\alpha\in[0,1]$, the new coefficients $(\Lambda^\alpha,\Gamma^\alpha)$ also satisfy Assumptions \ref{ass:4.1}-\ref{ass:4.2}, possessing the same Lipschitz constants, $\mu$, $\nu$, $M$, $A(\cdot)$,  $B(\cdot)$, $C(\cdot)$ as the original coefficients $(\Lambda,\Gamma)$.
	
	It is evident that when $\alpha=0$, FBS$\bigtriangleup $E \eqref{eq:4.10} can be rewritten in the form below:
	\begin{equation}\label{eq:4.12}
		\left\{\begin{aligned}
			x_{k+1}
			^{0}=&\bigg \{-\mu \Big[B_k^\top  Q(k,y'^{0}_{k+1},z'^{0}_{k+1})\Big]+\psi_k\bigg \}
			+\bigg \{-\mu\Big[ C_k^\top Q(k,y'^{0}_{k+1},z'^{0}_{k+1})\Big]+\gamma_k \}\omega_k\\
			y_{k}^0=&-\bigg\{-\nu\Big[A_k^{\top}P(k,x^{0}_k)\Big]+\varphi_k\bigg\},\quad k\in \mathbb{T}_{\infty},\\
			x_0^{0} =& -\mu M^\top M y^{0}_0+\xi.
		\end{aligned}\right.
	\end{equation}
	FBS$\bigtriangleup $E \eqref{eq:4.12} is in a decoupled form.  Moreover, when Assumption \ref{ass:4.2}-(i)-Case 1 is satisfied (i.e., $\mu>0$ and $\nu=0$), we are able to first solve for $y^0(\cdot)$ from the backward equation. After that, we can substitute $y^0(\cdot)$ into the forward equation and then solve for $x^0(\cdot)$. Likewise, when Assumption \ref{ass:4.2}-(i)-Case 2 is satisfied (i.e., $\mu=0$ and $\nu>0$), the forward and backward equations can be solved sequentially. Briefly speaking, when $\alpha=0$, under Assumptions \ref{ass:4.1}-\ref{ass:4.2}, FBS$\bigtriangleup $E \eqref{eq:4.12} has a unique solution $(x^0(\cdot),y^0(\cdot))\in N^{2,\rho}(\mathbb{T}_{\infty};\mathbb{R}^{2n} )$.
	
	Obviously, when $\alpha=1$ and $(\xi,\phi(\cdot) )$ vanish, FBS$\bigtriangleup $E \eqref{eq:4.10} and FBS$\bigtriangleup $E \eqref{eq:1.1} are completely identical.  Subsequently, we will demonstrate that if for some $\alpha_{0}\in[0,1)$, FBS$\bigtriangleup $E\eqref{eq:4.10} has a unique solution for any $\big(\xi ,\phi(\cdot)\big)\in \mathcal{H}^{2,\rho}(\mathbb{T}_{\infty})$, then there exists a fixed step length $\delta_0>0$ such that the same conclusion continues to hold for any $\alpha\in[\alpha_{0},\alpha_{0}+\delta_0]$. As soon as this is proved to hold, we can incrementally increase the parameter $\alpha$ until $\alpha=1$. This approach is known as the method of continuation, which was first introduced by Hu and Peng \cite{5}.
	
	Next, we prove a continuation lemma which is based on the prior estimate in Lemma \ref{lem:4.2}.
	\begin{lem}\label{lem:4.3}
		Let Assumptions  \ref{ass:4.1}-\ref{ass:4.2}, the conditions $ \ln(4L^2_1)< -\ln(6L^2_2)$ and  $ \ln(4L^2_1)<\rho <-\ln(6L^2_2)$ hold. If for some $\alpha_{0}\in[0,1)$, FBS$\bigtriangleup $E \eqref{eq:4.10} admits a unique solution $(x(\cdot),y(\cdot) )\in N^{2,\rho}(\mathbb{T}_{\infty};\mathbb{R}^{2n} )$ for any $(\xi,\phi(\cdot))\in \mathcal{H}^{2,\rho}(\mathbb{T}_{\infty})$, then there exists an absolute constant $\delta_0>0$ such that FBS$\bigtriangleup $E \eqref{eq:4.10} has a unique solution   for $\alpha=\alpha_{0}+\delta$ with $\delta\in(0,\delta_0]$ and $\alpha\le1$.
	\end{lem}
	\begin{proof}
		Let $\delta_0>0$ be determined below. For any $(x(\cdot),y(\cdot))\in N^{2,\rho}(\mathbb{T}_{\infty};\mathbb{R}^{2n} )$, $\theta_k=(x_k,y'_{k+1},z'_{k+1}) $ and $ \Theta_k=(X_k,Y'_{k+1},Z'_{k+1})$,we introduce
		the following FBS$\bigtriangleup $E with unknow $(X(\cdot),Y(\cdot))\in N^{2,\rho}(\mathbb{T}_{\infty};\mathbb{R}^{2n} )$ :
		\begin{equation}\label{eq:4.31}
			\left\{\begin{aligned}
				X_{k+1}  =&\Big\{-(1-\alpha_0)\mu\Big[B_k^{\top}Q(k,  Y_{k+1}', Z'_{k+1})\Big]+\alpha_0 b(k,\Theta_k)+\widetilde{\psi}_k\Big\}
				\\&+\Big\{-(1-\alpha_0)\mu\Big[C_k^{\top}Q(k, Y_{k+1}', Z'_{k+1})\Big]+\alpha_0 \sigma(k,\Theta_k)+\widetilde{\gamma}_k\Big\}\omega_k\\
				Y_k =&-\Big\{-(1-\alpha_0)\nu\Big[A_k^{\top}P(k,X_k)\Big]+\alpha_0f(k+1, \Theta_k)+\widetilde{\varphi}_k\Big\}, \quad k \in\mathbb{T}_{\infty},\\
				X_0=&\alpha_{0}\Lambda(Y_0)-(1-\alpha_{0})\mu M^{\top}MY_0+\widetilde{\xi},
			\end{aligned}\right.
		\end{equation}
		where
		\begin{equation}\label{eq:4.32}
			\left\{\begin{aligned}
				&\widetilde{\psi}_k=\psi_k +\delta b(k,\theta _k)+\delta \mu  \Big [B_k^\top Q(k, y_{k+1}', z'_{k+1})\Big ],\\
				&\widetilde{\gamma}_k=\gamma_k +\delta \sigma(k,\theta_k)+\delta \mu  \Big [C_k^\top Q(k, y_{k+1}', z'_{k+1}))\Big ],\\
				&\widetilde{\varphi}_k=\varphi_k+\delta f(k+1,\theta _k)+\delta \nu\Big[A_k^\top P(k,x_k)\Big],\\
				&\widetilde{\xi}=\xi+\delta\Lambda(y_0)+\delta\mu M^\top My_0,\\
				&Z_{k+1}=Y_{k+1}\omega_k,\\
				&Y'_{k+1}=[Y_{k+1}|\mathcal{F}_{k-1}],\\
				&Z'_{k+1}=[Y_{k+1}\omega_k|\mathcal{F}_{k-1}].
			\end{aligned}\right.
		\end{equation}
		Moreover, we also use the notation   $\widetilde{\phi}(\cdot) =(\widetilde{\varphi}(\cdot)  ,\widetilde{\psi}(\cdot) ,\widetilde{\gamma}(\cdot))$. Then it's straightforward to verify that $(\widetilde{\xi},\widetilde{\phi}(\cdot))\in\mathcal{H}^{2,\rho}(\mathbb{T}_{\infty})$. Based on our assumptions,
		the FBS$\bigtriangleup $E \eqref{eq:4.31} has a unique solution $(X(\cdot),Y(\cdot))\in N^{2,\rho}(\mathbb{T}_{\infty};\mathbb{R}^{2n} )$. Indeed, we have established a mapping
		\begin{equation}
			(X(\cdot),Y(\cdot))=\mathcal{T}_{\alpha _0+\delta }\big ((x(\cdot),y(\cdot))\big ): N^{2,\rho}(\mathbb{T}_{\infty};\mathbb{R}^{2n} )\to N^{2,\rho}(\mathbb{T}_{\infty};\mathbb{R}^{2n} ).\nonumber
		\end{equation}
		Below, we will prove that the mapping mentioned above is contractive when $\delta$ is sufficiently small.
		
		Let $(x(\cdot),y(\cdot) ),(\bar{x}(\cdot),\bar{y}(\cdot))\in N^{2,\rho}(\mathbb{T}_{\infty};\mathbb{R}^{2n} )$ and $(X(\cdot),Y(\cdot) )=\mathcal{T}_{\alpha _0+\delta }\big (x(\cdot),y(\cdot)\big )$, $(\bar{X}(\cdot),\bar{Y}(\cdot))=\mathcal{T}_{\alpha _0+\delta }\big (\bar{x}(\cdot),\bar{y}(\cdot)\big )$. Likewise, we denote $\widehat{x}(\cdot)=x(\cdot)-\bar{x}(\cdot),\widehat{y }(\cdot)=y(\cdot)-\bar{y}(\cdot),\widehat{X }(\cdot)=X(\cdot)-\bar{X}(\cdot),\widehat{Y }(\cdot)=Y(\cdot)-\bar{Y}(\cdot)$, etc. By making use of Lemma \ref{lem:4.2}, we obtain
		\begin{equation}
			\begin{aligned}
				&\quad\|\widehat{X}(\cdot)\|_{\rho}^2+\|\widehat{Y}(\cdot)\|_{\rho}^2\\
				&\le K\delta ^2\mathbb{E}\Bigg \{
				\Big|\Lambda(y_0)-\Lambda(\bar{y}_0)-\mu M^\top M\widehat{y_0}    \Big|^2 +\sum_{k=0}^{\infty}e^{-\rho k}\bigg |b(k,\theta_k)-b(k, \bar{\theta}_k )+\mu \big[B_k^\top Q(k,\widehat y_{k+1}',\widehat z_{k+1}')\big]\bigg |^2\\
				&\quad +\sum_{k=0}^{\infty}e^{-\rho k}\bigg | \sigma(k,\theta_k)-\sigma(k, \bar{\theta}_k)+\mu \big[C_k^\top Q(k,\widehat y_{k+1}',\widehat z_{k+1}')\big]\bigg |^2\\
				&\quad+\sum_{k=0}^{\infty}e^{-\rho (k+1)}\bigg|f(k+1,\theta_k)-f(k+1, \bar{\theta}_k )+\nu\big[A_k^\top P(k,\widehat{x}_k)\big] \bigg|^2
				\Bigg\}.\nonumber
			\end{aligned}
		\end{equation}
		Because of  the Lipschitz continuity of $(\Lambda,\Gamma)$ and the boundedness of $M,A(\cdot),B(\cdot),C(\cdot)$, there exists a new constant $K'>0$ which is not dependent on $\alpha_{0}$ and $\delta$ such that
		\begin{equation}
			\|(\widehat{X}(\cdot),\widehat{Y}(\cdot)) \|^2_{N^{2,\rho}(\mathbb{T}_{\infty};\mathbb{R}^{2n} )}\le K'\delta ^2 \|(\hat{x}(\cdot),\hat{y}(\cdot)) \|^2_{N^{2,\rho}(\mathbb{T}_{\infty};\mathbb{R}^{2n} )}.\nonumber
		\end{equation}
		By choosing $\delta_0=1/(2\sqrt{K'})$, when $\delta\in(0,\delta_0]$, we can find that the mapping $\mathcal{T}_{\alpha _0+\delta }$ is contractive. Consequently, the mapping $\mathcal{T}_{\alpha _0+\delta }$ has a unique fixed point,  which precisely is the unique solution to the FBS$\bigtriangleup $E \eqref{eq:4.10}. The proof is thus completed.
	\end{proof}
	We make a summary of the above analysis so as to provide the following proof.
	\begin{proof}[Proof of Theorem \ref{thm:4.1}]
		To begin with, the  solvability of FBS$\bigtriangleup $E \eqref{eq:1.1} in the space $N^{2,\rho}(\mathbb{T}_{\infty};\mathbb{R}^{2n} )$ is inferred from the unique solvability of FBS$\bigtriangleup $E \eqref{eq:4.12} and Lemma \ref{lem:4.3}. Secondly, within Lemma \ref{lem:4.2}, we get hold of the estimate \eqref{eq:4.8} in Theorem \ref{thm:4.1}. Finally, through selecting the coefficients $\big(\bar{\Lambda},\bar{\Gamma}\big)=(0,0)$, we obtain \eqref{eq:4.6} from \eqref{eq:4.8}. Regarding the uniqueness of the solution of FBS$\bigtriangleup $E \eqref{eq:1.1} , it can be obtained straightforwardly from  a prior estimate (\ref{eq:4.8}). And thus, the proof is completed.
	\end{proof}
	\section{Applications in linear quadratic problem}\label{sec:5}
	The objective of this section is to demonstrate the existence and uniqueness of solutions for stochastic Hamiltonian systems associated with specific types of linear-quadratic (LQ) optimal control problems. If it can be demonstrated that these stochastic Hamiltonian systems comply with the previously outlined domination-monotonicity conditions, then, in accordance with Theorem \ref{thm:4.1}, it can be concluded that their solutions exist and are unique. The primary incentive that led to this paper's study on the fully coupled infinite horizon FBS$\bigtriangleup $E (\ref{eq:1.1}) was, in reality, to figure out the solvability of these Hamiltonian systems.
	\subsection{Forward LQ stochastic control problem}\label{sec:5.1}
	
	In this section, we take into account the following linear forward S$\bigtriangleup $E control system:
	
	\begin{equation}\label{eq:5.1}
		\left\{\begin{aligned}
			x_{k+1}&=A_k x_k+B_k u_k+b_k+\left(C_k x_k+D_k u_k+\sigma_k\right) \omega_k, \quad k \in \mathbb{T}_{\infty} ,\\
			x_0&=\xi ,
		\end{aligned}\right.
	\end{equation}
	where $A(\cdot),C(\cdot) \in L_{\mathbb{F}}^{\infty}(\mathbb{T}_{\infty}; \mathbb{R}^{n\times n}),\ B(\cdot),D(\cdot) \in L_{\mathbb{F}}^{\infty}(\mathbb{T}_{\infty}; \mathbb{R}^{n\times m}),\  b(\cdot),\sigma(\cdot)\in L_{\mathbb{F}}^{2,\rho}\left(\mathbb{T}_{\infty}; \mathbb{R}^{n}\right). $  The vector $\xi \in L^{2} _{\mathcal{F}_{-1} }(\Omega ; \mathbb{R}^{n})$  is referred to as an initial state.  Here the stochastic process  $u(\cdot)=\left(u_{0}, u_{1}, u_{2}, \ldots\right)\in \mathbb{U}$ is called  an admissible control. Let $\rho$ be a constant satisfying
	\begin{equation}\label{eq:6.2}
		\rho>ln\left(4\left(\left\|A(\cdot)\right\|_{L_{\mathbb{F}}^{\infty}(\mathbb{T}_{\infty}; \mathbb{R}^{n\times n})}+\left\|C(\cdot)\right\|_{L_{\mathbb{F}}^{\infty}(\mathbb{T}_{\infty}; \mathbb{R}^{n\times n})}  \right)^{2}\right).
	\end{equation}
	According to Theorem \ref{thm:2.1}, when $\rho$ satisfies \eqref{eq:6.2}, S$\bigtriangleup $E (\ref{eq:5.1}) admits a unique solution $x(\cdot)\in L_{\mathbb{F}}^{2,\rho}\left(\mathbb{T}_{\infty}; \mathbb{R}^n\right)$.
	
	In the meantime, we put forth a quadratic criterion functional:
	\begin{equation}\label{eq:5.2}
		J\left(0 ; \xi, u(\cdot)\right)=\frac{1}{2} \mathbb{E} \left\{ \left\langle M \xi, \xi\right\rangle+\sum_{k=0}^{\infty}e^{-\rho k}\bigg[\left\langle Q_k x_k, x_k\right\rangle+\left\langle R_k u_k, u_k\right\rangle\bigg]\right\},
	\end{equation}
	where $M\in \mathbb{S}^n ,
	Q(\cdot) \in  L_{\mathbb{F}}^{\infty}(\mathbb{T}_{\infty}; \mathbb{S}^{n}) $ and  $R(\cdot) \in L_{\mathbb{F}}^{\infty}(\mathbb{T}_{\infty}; \mathbb{S}^{m})$.
	
	Next, we give the principal  questions to be addressed as follows:
	
	\textbf{Problem(FLQ).} The problem is to find a pair of initial state and admissible control $(\bar \xi, \bar u(\cdot)) \in  L^{2} _{\mathcal{F}_{-1} }(\Omega ; \mathbb{R}^{n}) \times \mathbb{U} $ such that

	\begin{equation}\label{eq:5.3}
		J(0; \bar\xi,  \bar{u}(\cdot))=\inf _{(\xi, u(\cdot)) \in L^{2} _{\mathcal{F}_{-1} }(\Omega ; \mathbb{R}^{n}) \times \mathbb{U}} J\left(0; \xi, u(\cdot)\right).
	\end{equation}
	Problem(FLQ) is considered to be  solvable if there exists $(\bar\xi,  \bar{u}(\cdot))$ above such that infimum of ceritrion function is achieved. In this case,  $\bar{u}(\cdot)$ is referred to as an  optimal control and $\bar{\xi}$ is called an optimal initial state.
	
	In addition, we put forward the following assumption:

	\begin{ass}\label{ass:5.1}
		(i) $M$ is positive definite;\\
		(ii) for any $(\omega ,k)\in \Omega \times  \mathbb{T}_{\infty}$, $Q_k$ is nonnegative definite, where $Q_0=0$;\\
		(iii) for any $(\omega ,k)\in \Omega \times  \mathbb{T}_{\infty}$, there exists a positive    constant $\delta$ such that  $R_k - \delta I_m \ge 0$ .
	\end{ass}
	
	With regard to the admissible control $(\bar \xi, \bar u(\cdot)) \in  L^{2} _{\mathcal{F}_{-1} }(\Omega ; \mathbb{R}^{n}) \times \mathbb{U} ,$ we introduce the stochastic Hamiltonian system of S$\bigtriangleup $E (\ref{eq:5.1}) as shown below:
	\begin{equation} \label{eq:5.4}
		\left\{\begin{aligned}
			\bar{x}_{k+1}=&A_k \bar{x}_k+B_k \bar{u}_k+b_k+\left\{C_k \bar{x}_k+D_k \bar{u}_k+\sigma_k\right\} \omega_k, \\
			\bar{y}_k=&e^{-\rho}A_k^{\top} \bar{y}_{k+1}^{\prime}+e^{-\rho}C_k^{\top} \bar{z}_{k+1}^{\prime}+Q_k \bar{x}_k, \quad k \in \mathbb{T}_{\infty}, \\
			\bar{x}_0=&-M^{-1} \bar{y}_0  ,\\
			0=&e^{-\rho}B_k^{\top} \bar{y}_{k+1}^{\prime}+e^{-\rho}D_k^{\top} \bar{z}_{k+1}^{\prime}+ R_k\bar{u}_k .
		\end{aligned}\right.
	\end{equation}
	Evidently, the Hamiltonian system \eqref{eq:5.4} is characterized by a FBS$\bigtriangleup $E that meets Assumptions \ref{ass:4.1}-\ref{ass:4.2}.
	\begin{rmk}
		Similar to the analysis on S$\bigtriangleup $E (\ref{eq:5.1}), we can also set a condition
		\begin{equation}\label{eq:6.6}
			\rho>ln\left(6\left(\left\|A(\cdot)\right\|_{L_{\mathbb{F}}^{\infty}(\mathbb{T}_{\infty}; \mathbb{R}^{n\times n})}+\left\|C(\cdot)\right\|_{L_{\mathbb{F}}^{\infty}(\mathbb{T}_{\infty}; \mathbb{R}^{n\times n})}  \right)^{2}\right).
		\end{equation}
		Then, by Theorem \ref{thm:2.2}, the BS$\bigtriangleup $E in Hamiltonian system \eqref{eq:5.4} admits a unique solution $y(\cdot)\in L_{\mathbb{F}}^{2,\rho}\left(\mathbb{T}_{\infty}; \mathbb{R}^n\right)$.
	\end{rmk}
	In order to comply the condition in Theorem \ref{thm:4.1}, we can let \eqref{eq:6.6} hold,
	so that  \eqref{eq:6.2} also holds.
	By making use of Theorem \ref{thm:4.1}, we can obtain the following Theorem.
	\begin{thm}\label{thm:5.1}
		Let Assumption \ref{ass:5.1} and condition \eqref{eq:6.6} hold.  Then the Hamiltonian system \eqref{eq:5.4} admits a unique solution $(\bar{x}(\cdot),\bar{y}(\cdot)) \in N^{2,\rho}(\mathbb{T}_{\infty} ;\mathbb{R}^{2n} )$. Moreover,\\
		\begin{equation}\label{eq:5.5}
			\bar{\xi}=-M^{-1} \bar{y}_0
		\end{equation}
		is the unique optimal initial state and
		\begin{equation}\label{eq:5.6}
			\bar{u}_k=-R_k^{-1}\left(e^{-\rho}B_k^{\top} \bar{y}_{k+1}^{\prime}+e^{-\rho}D_k^{\top} \bar{z}_{k+1}^{\prime}\right)
		\end{equation}
		is the unique optimal control of Problem (FLQ).
	\end{thm}
	\begin{proof}
		
		From the last equation of Hamiltonian system (\ref{eq:5.4}), we get \eqref{eq:5.6}.
		By substituting equation (\ref{eq:5.6}) into the Hamiltonian system (\ref{eq:5.4}), we derive a FBS$\bigtriangleup $E.
		It is easy to check that  the coefficients of this FBS$\bigtriangleup $E satisfy  Assumptions \ref{ass:4.1}-\ref{ass:4.2}, so Theorem \ref{thm:4.1} ensures the existence of a unique solution to this FBS$\bigtriangleup $E, which in turn guarantees a unique solution to the Hamiltonian system \eqref{eq:5.4}.
		
		Take into account $(\bar \xi, \bar u(\cdot))\in L^{2} _{\mathcal{F}_{-1} }(\Omega ; \mathbb{R}^{n}) \times \mathbb{U}$ that satisfies conditions (\ref{eq:5.5}) and (\ref{eq:5.6}). At present, we will prove the optimality as well as the uniqueness of the pair (\ref{eq:5.5}) and (\ref{eq:5.6}). We initiate our proof by establishing the existence of an optimal initial state and an optimal control.  Let $(\xi,u(\cdot)) \in L^{2} _{\mathcal{F}_{-1} }(\Omega ; \mathbb{R}^{n}) \times \mathbb{U}$ be an arbitrary set of initial states and admissible controls, and let $x(\cdot)$ denotes the state process corresponding to $(\xi,u(\cdot))$. We will then proceed to analyze the difference in the criterion function between these two different states:
		\begin{equation} \label{eq:5.8}
			\begin{aligned}
				& J\left(0 ; \xi, u(\cdot)\right)-J(0;\bar{\xi}, \bar{u}(\cdot)) \\
				=&\frac{1}{2} \mathbb{E} \left \{\left\langle M \xi, \xi \right\rangle-\langle M \bar{\xi}, \bar{\xi}\rangle+\sum_{k=0}^{\infty}e^{-\rho k}\bigg[\left\langle Q_k x_k, x_k\right\rangle-\left\langle Q_k \bar{x}_k, \bar{x}_k\right\rangle+\left\langle R_k u_k, u_k\right\rangle-\left\langle R_k \bar{u}_k, \bar{u}_k\right\rangle\bigg]\right\} \\
				=&\frac{1}{2} \mathbb{E} \left \{\left\langle M\left(\xi-\bar{\xi} \right),\xi-\bar{\xi} \right\rangle+\sum_{k=0}^{\infty}e^{-\rho k}\bigg[\left\langle Q_k\left(x_k-\bar{x}_k\right), x_k-\bar{x}_k\right\rangle+\left\langle R_k\left(u_k-\bar{u}_k\right), u_k-\bar{u}_k\right\rangle\bigg]\right\}+\Delta_1,
			\end{aligned}
		\end{equation}
		where
		
		\begin{equation}\label{eq:5.9}
			\begin{aligned}
				\Delta_1=\mathbb{E} \Bigg \{  \left\langle M \bar{\xi}, \xi-\bar{\xi} \right\rangle
				+\sum_{k=0}^{\infty}e^{-\rho k}\left[\left\langle Q_k \bar{x}_k, x_k-\bar{x}_k\right\rangle+\left\langle R_k \bar{u}_k, u_k-\bar{u}_k\right\rangle\right] \Bigg \}.
			\end{aligned}
		\end{equation}
		Besides, according to the expression for the initial value in \eqref{eq:5.4}, we find that
		
		\begin{equation} \label{eq:5.10}
			\begin{aligned}
				&\mathbb{E} \Bigg \{  \left\langle M \bar{\xi}, \xi-\bar{\xi} \right\rangle+e^{-\rho N}\left\langle  \bar{y}_N, x_N-\bar{x}_N\right\rangle\bigg\}
				\\=&\mathbb{E} \Bigg \{e^{-\rho N}\left\langle x_N-\bar{x}_N, \bar{y}_N\right\rangle-\left\langle \xi-\bar{\xi}, \bar{y}_0\right\rangle \Bigg \} \\
				=&   \mathbb{E} \Bigg \{\sum_{k=0}^{N-1}e^{-\rho k}\bigg[e^{-\rho}\left\langle x_{k+1}-\bar{x}_{k+1}, \bar{y}_{k+1}\right\rangle-\left\langle x_k-\bar{x}_k, \bar{y}_k\right\rangle   \bigg]\bigg\}
				\\= &\mathbb{E}\left\{\sum_{k=0}^{N-1}e^{-\rho k}\bigg[e^{-\rho}\left \langle A_k\left(x_k -\bar{x}_k \right)+B_k\left(u_k-\bar{u}_k\right)+\left[ C_k\left(x_k -\bar{x}_k\right)+D_k\left(u_k-\bar{u}_k\right)\right] \omega_k, \bar{y}_{k+1}\right\rangle\right. \\
				& -\left\langle x_k-\bar{x}_k, e^{-\rho}A_k^{\top} \bar{y}_{k+1}^{\prime}+e^{-\rho}C_k^{\top} \bar{z}_{k+1}^{\prime}+Q_k \bar{x}_k\right\rangle \bigg]\Bigg \} \\
				=&\mathbb{E} \Bigg \{ \sum_{k=0}^{N-1}e^{-\rho k}\bigg[\left\langle e^{-\rho}B_k^{\top} \bar{y}_{k+1}^{\prime}+e^{-\rho}D_k^{\top} \bar{z}_{k+1}^{\prime}, u_k-\bar{u}_k\right\rangle-\left\langle Q_k \bar{x}_k , x_k-\bar{x}_k\right\rangle\bigg]\Bigg \}.
		\end{aligned}\end{equation}
		Letting $N\rightarrow\infty$ on both sides of \eqref{eq:5.10}, we have
		\begin{equation} \label{eq:5.11}
			\begin{aligned}
				\mathbb{E} \left\langle M \bar{\xi}, \xi-\bar{\xi} \right\rangle
				=\mathbb{E} \Bigg \{ \sum_{k=0}^{\infty}e^{-\rho k}\bigg[\left\langle e^{-\rho} B_k^{\top} \bar{y}_{k+1}^{\prime}+e^{-\rho}D_k^{\top} \bar{z}_{k+1}^{\prime}, u_k-\bar{u}_k\right\rangle-\left\langle Q_k \bar{x}_k , x_k-\bar{x}_k\right\rangle\bigg]\Bigg \}.
		\end{aligned}\end{equation}
		Substituting \eqref{eq:5.11} into \eqref{eq:5.9}, we obtain that
		\begin{equation}\label{eq:5.12}
			\begin{aligned}
				\Delta_1=& \mathbb{E} \Bigg \{ \sum_{k=0}^{\infty}e^{-\rho k}\left\langle e^{-\rho}B_k^{\top} \bar{y}_{k+1}^{\prime}+e^{-\rho}D_k^{\top} \bar{z}_{k+1}^{\prime}+R_k \bar{u}_k, u_k-\bar{u}_k\right\rangle \Bigg \}=0,
			\end{aligned}
		\end{equation}
		where we have used  $$0=e^{-\rho}B_k^{\top} \bar{y}_{k+1}^{\prime}+e^{-\rho}D_k^{\top} \bar{z}_{k+1}^{\prime}+ R_k\bar{u}_k $$ in  \eqref{eq:5.4}. After that, substituting \eqref{eq:5.12} into \eqref{eq:5.8} and making use of the positive definiteness of the relevant coefficient matrices in Assumption \ref{ass:5.1}, we can easily check that $$J\left(0 ; \xi, u(\cdot)\right)-J\left(0 ; \bar{\xi}, \bar{u}(\cdot)\right) \ge 0.$$  Since ($\xi ,u(\cdot)$) is an arbitrary pair of initial state and admissible control, we obtain the optimality of ($\bar{\xi},\bar{u}(\cdot)$).
		
		Subsequently, we proceed to prove the uniqueness.
		
		Assume that $(\tilde{x}_0, \tilde{v}(\cdot))\in L^{2} _{\mathcal{F}_{-1} }(\Omega ; \mathbb{R}^{n}) \times \mathbb{U}$ denotes another optimal initial state and control pair. The corresponding optimal state process for $(\tilde{x}_0, \tilde{v}(\cdot))$ is denoted by $\tilde{x}(\cdot)$. This means that the condition $J(0 ; \bar{\xi}, \bar{u}(\cdot)) = J(0 ; \tilde{x}_0, \tilde{v}(\cdot))$ holds. By referring back to (\ref{eq:5.8}), we can deduce:
		\begin{equation}\label{eq:5.13}
			\begin{aligned}
				0=& J(0 ; \bar{\xi}, \bar{u}(\cdot)) - J(0 ; \tilde{x}_0, \tilde{v}(\cdot))
				\\=&\frac{1}{2} \mathbb{E} \left \{\left\langle M\left(\bar{\xi}-\tilde{x}_0\right), \bar{\xi}-\tilde{x}_0\right\rangle+\sum_{k=0}^{\infty}e^{-\rho k}\bigg[\left\langle Q_k\left(\bar{x}_k-\tilde{x}_k\right), \bar{x}_k-\tilde{x}_k\right\rangle+\left\langle R_k\left(\bar{u}_k-\tilde{v}_k\right), \bar{u}_k-\tilde{v}_k\right\rangle\bigg]\right\} \\
				\geq& \frac{1}{2} \mathbb{E} \left\{\left\langle M\left(\bar{\xi}-\tilde{x}_0\right), \bar{\xi}-\tilde{x}_0\right\rangle+\sum_{k=0}^{\infty}e^{-\rho k}\bigg[\left\langle R_k\left(\bar{u}_k-\tilde{v}_k\right), \bar{u}_k-\tilde{v}_k\right\rangle\bigg]\right\}.
			\end{aligned}
		\end{equation}
		According to positive definiteness of $M$ and $R_k$, we have $\bar{\xi} = \tilde{x}_0$ and $\bar{u}_k=\tilde{v}_k$. The uniqueness is proved. Thus, the proof is finished.
	\end{proof}
	\begin{rmk}
		The theoretical optimal control law $\bar{u}_{k}$ in Theorem \ref{thm:5.1} decomposes into two components: a state-feedback term and an affine correction term:
		$$
		\bar{u}_{k} = K_{k} x_{k} + \eta_{k}
		$$
		where $K_{k}$ is the feedback gain matrix obtained from Riccati solutions and $\eta_{k}$ is a bias term governed by a BS$\bigtriangleup$E, induced by the system's affine dynamics and noise structure.
		
		The derivation process requires systematic decoupling of the Hamiltonian system through Riccati operators, involving sequential matrix operations. As the primary focus of this work is establishing the solvability and structure of the fully coupled  FBS$\bigtriangleup$E, the explicit formulation of $K_{k}$ and its numerical implementation will be detailed in subsequent publications.
	\end{rmk}
	\subsection{Backward LQ stochastic control problem}\label{sec:5.2}
	
	Within this section, we focus on the following linear BS$\bigtriangleup $E control system:
	
	\begin{equation}
		\begin{aligned} \label{eq:5.14}
			y_k=A_k y'_{k+1}+B_k z'_{k+1}+C_k v_k+\alpha_k,  \quad k\in \mathbb{T}_{\infty}
		\end{aligned}
	\end{equation}
	where $A(\cdot),B(\cdot)\in L_{\mathbb{F}}^{\infty}(\mathbb{T}_{\infty}; \mathbb{R}^{n\times n}),
	C(\cdot)\in L_{\mathbb{F}}^{\infty}(\mathbb{T}_{\infty}; \mathbb{R}^{n\times m}), \alpha(\cdot) \in L_{\mathbb{F}}^{2,\rho}\left(\mathbb{T}_{\infty}; \mathbb{R}^{n}\right). $ The process $v(\cdot)$ is called an admissible control and $v(\cdot)=\left(v_{0}, v_{1},v_{2}, \ldots \right)  \in \mathbb{U}$. Let $\rho$ be a constant satisfying
	\begin{equation}  \label{eq:6.3}
		\begin{aligned}
			\rho<-ln\left(6\left(\left\|A(\cdot)\right\|_{L_{\mathbb{F}}^{\infty}(\mathbb{T}_{\infty}; \mathbb{R}^{n\times n})}+\left\|B(\cdot)\right\|_{L_{\mathbb{F}}^{\infty}(\mathbb{T}_{\infty}; \mathbb{R}^{n\times n})}  \right)^{2}\right).
		\end{aligned}
	\end{equation}
	According to Theorem \ref{thm:2.2}, when $\rho$ satisfies \eqref{eq:6.3}, BS$\bigtriangleup $E (\ref{eq:5.14}) admits a unique solution $y(\cdot)\in L_{\mathbb{F}}^{2,\rho}\left(\mathbb{T}_{\infty}; \mathbb{R}^n\right)$.

	In the meantime, we put forth a quadratic criterion function:
	\begin{equation}  \label{eq:5.15}
		\begin{aligned}
			J(v) =\frac{1}{2} \mathbb{E}\left\{\left\langle M y_0, y_0\right\rangle+\sum_{k=0}^{\infty}e^{-\rho k}\bigg[\left\langle Q_{k} y'_{k+1} ,y'_{k+1}\right\rangle+\left\langle L_k z'_{k+1}, z'_{k+1}\right\rangle+\left\langle R_k v_k, v_k\right\rangle \bigg]\right\},
		\end{aligned}
	\end{equation}
	where $M\in \mathbb{S}^n ,
	Q(\cdot),L(\cdot) \in    L_{\mathbb{F}}^{\infty}(\mathbb{T}_{\infty}; \mathbb{S}^n) $ and  $R(\cdot) \in L_{\mathbb{F}}^{\infty}(\mathbb{T}_{\infty}; \mathbb{S}^m) $.
	
	Next, we give the principal  question to be addressed as follows:
	
	\textbf{Problem (BLQ).} We need to find an admissible control $\bar{v}(\cdot) \in \mathbb{U}$ such that

	\begin{equation}\label{eq:5.16}
		J( \bar{v}(\cdot))=\inf _{v(\cdot) \in \mathbb{U}} J(v(\cdot)).
	\end{equation}
	
	Problem (BLQ) is considered to be solvable if there exists $\bar{v}(\cdot)$ above such that infimum of the criterion function \eqref{eq:5.15} is achieved. In this case, $\bar{v}(\cdot)$ is referred to as an optimal control .
	
	In addition, we put forward the following assumption:
	\begin{ass}\label{ass:5.2}
		(i) $M$ is positive definite;\\
		(ii) for any $(\omega ,k)\in \Omega \times  \mathbb{T}_{\infty}$, $Q_k$ and $L_k$ are nonnegative definite;\\
		(iii) for any $(\omega ,k)\in \Omega \times  \mathbb{T}_{\infty}$, there exists a positive constant $\delta$ such that  $R_k - \delta I_m \ge 0$ .
	\end{ass}
	Hamiltonian system for Problem(BLQ) is presented as follows:
	\begin{equation}\label{eq:5.17}
		\left\{\begin{aligned}
			\bar x_{k+1}&=e^{\rho }A_k^{\top} \bar x_k-e^{\rho}Q_k \bar y'_{k+1}+e^{\rho }\left(B_k^{\top}\bar  x_k-L_k \bar z'_{k+1}\right) \omega_k , \quad k\in
			\mathbb{T}_{\infty} ,\\
			\bar y_k&=A_k \bar y'_{k+1}+B_k \bar z'_{k+1}+C_k \bar v_k+\alpha_k , \\
			\bar x_0&=-M \bar y_0 ,\\
			0&=C_k^{\top} \bar x_k-R_k \bar v_k.
		\end{aligned}\right.
	\end{equation}
	\begin{rmk}
		Similar to before, we can also set a condition
		\begin{equation}\label{eq:7.6}
			\rho<-ln\left(4\left(\left\|A(\cdot)\right\|_{L_{\mathbb{F}}^{\infty}(\mathbb{T}_{\infty}; \mathbb{R}^{n\times n})}+\left\|B(\cdot)\right\|_{L_{\mathbb{F}}^{\infty}(\mathbb{T}_{\infty}; \mathbb{R}^{n\times n})}  \right)^{2}\right).
		\end{equation}
		Then, by Theorem \ref{thm:2.1}, the S$\bigtriangleup $E in Hamiltonian system \eqref{eq:5.17} admits a unique solution $x(\cdot)\in L_{\mathbb{F}}^{2,\rho}\left(\mathbb{T}_{\infty}; \mathbb{R}^n\right)$.
	\end{rmk}
	In order to comply the condition in Theorem \ref{thm:4.1}, we can let \eqref{eq:6.3} hold,
	so that  \eqref{eq:7.6} also holds.
	By making use of Theorem \ref{thm:4.1}, we can obtain the following Theorem.
	\begin{thm}\label{thm:5.2}
		Let Assumption \ref{ass:5.2} and condition \eqref{eq:6.3} hold.  Then the above Hamiltonian system \eqref{eq:5.17} admits a unique solution $(\bar {x}(\cdot),\bar{y}(\cdot)) \in N^{2,\rho}(\mathbb{T}_{\infty} ;\mathbb{R}^{2n} )$. Moreover,
		\begin{equation} \label{eq:5.18}
			\bar v_k=R^{-1}_k C_k^{\top} \bar x_k
		\end{equation}
		is the unique optimal control of Problem (BLQ).
	\end{thm}
	\begin{proof}
		From the final relationship in \eqref{eq:5.17}, we get the relationship (\ref{eq:5.18}). Upon substituting (\ref{eq:5.18}) into the Hamiltonian system (\ref{eq:5.17}), an FBS$\bigtriangleup $E is derived. Then, it is easy to see that the coefficients of this FBS$\bigtriangleup $E satisfy Assumptions \ref{ass:4.1}-\ref{ass:4.2}. Consequently, according to Theorem \ref{thm:4.1}, there exists a unique solution $(\bar x(\cdot),\bar y(\cdot) ) \in N^{2,\rho}(\mathbb{T}_{\infty} ;\mathbb{R}^{2n} )$ for this FBS$\bigtriangleup $E. This indicates that there is a unique solution to the Hamiltonian system (\ref{eq:5.17}).
		
		Now we start to prove  (\ref{eq:5.18}) is an unique optimal control.
		Suppose $v(\cdot) \in  \mathbb{U} $ is any given  admissible controls, and let $ y(\cdot) $ denotes the corresponding state process associated with $v(\cdot)$. We will then take into account the difference between the cost functionals in these two distinct states:
		\begin{equation} \label{eq:5.19}
			\begin{aligned}
				& J( v(\cdot))-J( \bar{v}(\cdot)) \\
				& =\frac{1}{2} \mathbb{E}\Bigg\{ \left\langle M y_0, y_0\right\rangle-\left\langle M \bar{y}_0, \bar{y}_0\right\rangle+\sum_{k=0}^{\infty}e^{-\rho k}\bigg[\left\langle Q_k {y}'_{k+1}, {y}'_{k+1}\right\rangle -\left\langle Q_k \bar{y}'_{k+1}, \bar{y}'_{k+1}\right\rangle\\
				& \left. \quad +\left\langle L_k {z}'_{k+1}, {z}'_{k+1}\right\rangle -\left\langle L_k \bar{z}'_{k+1}, \bar{z}'_{k+1}\right\rangle+\left\langle R_k v_k, v_k\right\rangle -\left\langle R_k \bar{v}_k, \bar{v}_k\right\rangle\bigg]\right\} \\
				& =\frac{1}{2}\mathbb{E}\Bigg\{\left\langle M\left(y_0-\bar{y}_0\right), y_0-\bar{y}_0\right\rangle +\sum_{k=0}^{\infty}e^{-\rho k}\bigg[\left\langle Q_k\left({y}'_{k+1}-\bar{y}'_{k+1}\right), {y}'_{k+1}-\bar{y}'_{k+1}\right\rangle\\
				&\left. \quad +\left\langle L_k\left({z}'_{k+1}-\bar{z}'_{k+1}\right), {z}'_{k+1}-\bar{z}'_{k+1}\right\rangle + \left\langle R_k(v_k-\bar{v}_k),v_k-\bar{v}_k\right\rangle \bigg]\right\}+\Delta_2,
			\end{aligned}
		\end{equation}
		where
		\begin{equation}\label{eq:5.20}
			\Delta_2=\mathbb{E} \left\{\left\langle M \bar{y}_0, y_0-\bar{y}_0\right\rangle+\sum_{k=0}^{\infty}e^{-\rho k}\left[\left\langle Q_k {\bar{y}}'_{k+1}, {y}'_{k+1}-\bar{y}'_{k+1}\right\rangle+
			\left\langle L_k {\bar{z}}'_{k+1},  {z}'_{k+1}-\bar{z}'_{k+1} \right\rangle+\left\langle R_k \bar{v}_k, v_k-\bar{v}_k\right\rangle\right]\right\}.
		\end{equation}
		Moreover,  from the relationships in (\ref{eq:5.17}), we can get that
		
		\begin{equation} \label{eq:5.21}
			\begin{aligned}
				&\mathbb{E}\left[\left\langle M \bar{y}_0, y_0-\bar{y}_0\right\rangle+e^{-\rho N}\left\langle\bar{x}_N, y_N-\bar{y}_N\right\rangle \right]
				\\=&\mathbb E\bigg\{e^{-\rho N}\left\langle \bar{x}_N, y_N-\bar{y}_N\right\rangle- \left\langle \bar{x}_0, y_0-\bar{y}_0\right\rangle\bigg\}
				\\=& \mathbb E\bigg\{\sum_{k=0}^{N -1}e^{-\rho k}\bigg[e^{-\rho}\left\langle \bar{x}_{k+1}, y_{k+1}-\bar{y}_{k+1}\right\rangle-\left\langle \bar{x}_k, y_{k}-\bar{y}_{k}\right\rangle\bigg]
				\bigg\}\\
				=&\mathbb{E}\bigg\{ \sum_{k=0}^{N-1}e^{-\rho k}\bigg[
				e^{-\rho}\left \langle e^{\rho }A_k^{\top} \bar x_k-e^{\rho }Q_k \bar y'_{k+1}+e^{\rho }\left(B_k^{\top}\bar  x_k-L_k \bar z'_{k+1}\right) \omega_k ,y_{k+1}-\bar{y}_{k+1}\right\rangle\\
				&\quad - \left\langle \bar{x}_k ,A_k (y'_{k+1}-\bar y'_{k+1})+B_k (z'_{k+1}-\bar z'_{k+1})+C_k (v_k-\bar v_k)\right\rangle
				\bigg]\bigg\}\\
				=&\mathbb{E}\bigg\{\sum_{k=0}^{N-1}e^{-\rho k}\bigg[-\left\langle Q_k \bar{y}'_{k+1}, y'_{k+1}-\bar{y}'_{k+1}\right\rangle-\left\langle L_k \bar{z}'_{k+1}, z'_{k+1}-\bar{z}'_{k+1}\right\rangle-\left\langle C_k^{\top} \bar{x}_k, v_k-\bar{v}_k\right\rangle\bigg] \bigg\}\\
				=&\mathbb{E}\bigg\{\sum_{k=0}^{N-1}e^{-\rho k}\bigg[-\left\langle Q_k \bar{y}'_{k+1}, y'_{k+1}-\bar{y}'_{k+1}\right\rangle-\left\langle L_k \bar{z}'_{k+1}, z'_{k+1}-\bar{z}'_{k+1}\right\rangle-\left\langle R_k^{\top} \bar{v}_k, v_k-\bar{v}_k\right\rangle\bigg]\bigg\}.
			\end{aligned}
		\end{equation}
		Let $N\rightarrow\infty$ on both sides of \eqref{eq:5.21}. We have
		\begin{equation} \label{eq:5.22}
			\begin{aligned}
				\mathbb{E}\left[\left\langle M \bar{y}_0, y_0-\bar{y}_0\right\rangle \right]
				=\mathbb{E}\bigg\{\sum_{k=0}^{\infty}e^{-\rho k}\left[-\left\langle Q_k \bar{y}'_{k+1}, y'_{k+1}-\bar{y}'_{k+1}\right\rangle-\left\langle L_k \bar{z}'_{k+1}, z'_{k+1}-\bar{z}'_{k+1}\right\rangle
				-\left\langle R_k^{\top} \bar{v}_k, v_k-\bar{v}_k\right\rangle\right]\bigg\}.
			\end{aligned}
		\end{equation}
		Consequently, $\Delta_2 = 0$. Then, based on \eqref{eq:5.19} and by taking into consideration the positive definiteness of the relevant coefficient matrices in Assumption \ref{ass:5.2}, it can be readily demonstrated that $J( v(\cdot)) - J( \bar{v}(\cdot)) \ge 0$. Since  $v(\cdot)$ is an arbitrary admissible control, this verifies the optimality of $\bar{v}(\cdot)$.

		Now we continue to prove the uniqueness. Suppose that there exists another optimal control  $\tilde{u}(\cdot) \in \mathbb{U}$ aside from $\bar{v}(\cdot)$. Let $\tilde{y}(\cdot)$ be the state process corresponding to $\tilde{u}(\cdot)$. This means that $J(\tilde{u}(\cdot)) = J(\bar{v}(\cdot))$. Then, by referring back to (\ref{eq:5.19}) and \eqref{eq:5.22}, we can infer the following:
		\begin{equation}\label{eq:5.23}
			\begin{aligned}
				0= &J( \tilde{u}(\cdot)) - J( \bar{v}(\cdot))
				\\=&\frac{1}{2}\mathbb{E}\Bigg\{\left\langle M\left(\tilde{y}_0-\bar{y}_0\right), \tilde{y}_0-\bar{y}_0\right\rangle +\sum_{k=0}^{\infty}\bigg[\left\langle Q_k\left(\tilde{y}'_{k+1}-\bar{y}'_{k+1}\right), \tilde{y}'_{k+1}-\bar{y}'_{k+1}\right\rangle\\
				&\left. \quad +\left\langle L_k\left(\tilde{z}'_{k+1}-\bar{z}'_{k+1}\right), \tilde{z}'_{k+1}-\bar{z}'_{k+1}\right\rangle + \left\langle R_k(\tilde{u}_k -\bar{v}_k),\tilde{u}_k -\bar{v}_k\right\rangle \bigg]\right\}\\
				\geqslant & \frac{1}{2} \mathbb{E}\Bigg\{\left\langle M\left(\tilde{y}_0-\bar{y}_0\right), \tilde{y}_0-\bar{y}_0\right\rangle+\sum_{k=0}^{\infty}\bigg[\left\langle R_k( \tilde{u}_k-\bar{v}_k), \tilde{u}_k-\bar{v}_k\right\rangle\bigg]\Bigg\}.
			\end{aligned}
		\end{equation}
		According to positive definiteness of $M$ and $R_k$, we get $\tilde{u}_k=\bar{v}_k$. It has been proved that the uniqueness holds.
	\end{proof}
	
	\section*{Declarations}
	The authors have not disclosed any competing interests.
	\bibliographystyle{amsplain}
	
\end{document}